\documentclass[]{interact}
\usepackage{lmodern}
\usepackage{epstopdf}
\usepackage[caption=false]{subfig}

\usepackage[numbers,sort&compress]{natbib}
\bibpunct[, ]{[}{]}{,}{n}{,}{,}
\usepackage{algorithm}
\usepackage{algorithmic}
\theoremstyle{plain}
\newtheorem{theorem}{Theorem}[section]
\newtheorem{lemma}[theorem]{Lemma}

\theoremstyle{definition}
\newtheorem{definition}[theorem]{Definition}

\theoremstyle{remark}
\newtheorem{remark}{Remark}

\newtheorem{properties}[theorem]{Properties}

\newcommand{\mc}{\mathcal}

\begin{document}


\title{Quaternion Nonlinear Transform-Induced Nuclear Norm for Low-Rank Tensor Completion}

\author{
\name{
Biswarup Karmakar\thanks{CONTACT Biswarup Karmakar. Email: biswarupk@iisc.ac.in; Ratikanta Behera: ratikanta@iisc.ac.in} \and
Ratikanta Behera
}
\affil{
Department of Computational and Data Sciences, Indian Institute of Science, Bangalore, 560012, India
}
}
\maketitle

\begin{abstract}
Tensor completion has emerged as a powerful framework for recovering missing data in multidimensional signals by exploiting low-rank tensor structures. Among existing approaches, linear transform–based tensor nuclear norm (TNN) methods have achieved considerable success by enforcing low-rankness on transformed frontal slices. However, the low-rank structure revealed by linear transforms remains inherently limited. To better capture intrinsic correlations, nonlinear transform–based TNN (NTTNN) models have been proposed, significantly enhancing low-rank representation through composite transforms. Despite their effectiveness, existing NTTNN methods are restricted to real-valued tensors and fail to model quaternion-valued data, which are essential for preserving inter-channel dependencies in color images and videos. Extending nonlinear TNN models to the quaternion domain is challenging due to the non-commutativity of quaternion multiplication and the complexity of quaternion singular value decomposition. To address the limitations encountered in prior works, we propose a quaternion nonlinear transform–induced tensor nuclear norm (QNTTNN) via a real embedding of quaternions, enabling tractable nuclear norm definitions and efficient optimization. Building upon QNTTNN, we formulate a quaternion tensor completion model and develop a proximal alternating minimization algorithm with rigorous convergence guarantees. Extensive experiments on benchmark color video inpainting datasets validate the superior performance of the proposed method over existing approaches.
\end{abstract}


\section{Introduction}
Tensors form a natural mathematical structure for modeling data with multiple dimensions. So far, tensor-based models have been widely used in many fields, including signal processing~\cite{tokcan2026TDforSP}, data mining~\cite{papalexakis2016tensors}, machine learning~\cite{panagakis2021tensor,sidiropoulos2017tensor}, biomedical imaging~\cite{becker2015brain}, and recommendation systems~\cite{frolov2017tensor}. A color image can be viewed as a third-order tensor, while a color video can be represented as a fourth-order tensor by including the time dimension. Color video inpainting aims to fill in missing or damaged regions in RGB video sequences while keeping the spatial structure, temporal smoothness, and color consistency. A common and effective approach is to reshape the video data as a third-order tensor and recover it using low-rank tensor completion methods, which have shown strong performance in reconstructing high-dimensional visual data \cite{ding2019lowranktensor,Qin2022LRHOTC}. Given an incomplete tensor $\mathcal{M}\in \mathbb{R}^{n_1\times n_2\times n_3}$
with observed index set $\Omega$, tensor completion aims to recover a low-rank
tensor $\mathcal{X}\in \mathbb{R}^{n_1 \times n_2 \times n_3}$ such that
\begin{equation}
\min_{\mathcal{X}\in \mathbb{R}^{n_1 \times n_2 \times n_3}} \ \mathrm{rank}(\mathcal{X})
\quad \text{s.t.} \quad
\mathcal{P}_{\Omega}(\mathcal{X}) = \mathcal{P}_{\Omega}(\mathcal{M}),
\end{equation}
where $\mathcal{P}_{\Omega}$ denotes the projection operator onto the observed
entries. This rank minimization problem is NP-hard, and various strategies~\cite{lee2010atomic,jin2015penaltydecomp} have
been proposed to address it. A widely used approach replaces the rank function
with convex surrogates such as the tensor nuclear norm (TNN), typically defined
via the tensor singular value decomposition (t-SVD)~\cite{kilmer13,kernfeld2015,zhang2017exact}. These methods have also been extended to symmetric nonnegative tensors~\cite{duan2019iterative} via nuclear norm minimization under partial observations. These models have been extensively studied for grayscale video, hyperspectral imaging, and multidimensional signal recovery.

However, for color image and video data, real-valued tensor models typically
process RGB channels independently or as weakly coupled modes, which limits
their ability to capture intrinsic inter-channel correlations~\cite{ell2007hypercomplex,zhang1997quaternions}. Quaternion representation addresses this limitation by jointly encoding the RGB channels
into a single quaternion entity, enabling unified and correlation-aware color
modeling. Motivated by these advantages, tensor completion has been extended from the
real-valued domain to the quaternion domain for color image and video recovery.
Let $\mathcal{M} \in \mathbb{Q}^{n_1 \times n_2 \times n_3}$ denote an incomplete
quaternion-valued tensor, where $\mathbb{Q}$ is the quaternion field and
$\Omega$ is the index set of observed entries. Quaternion tensor completion
seeks a low-rank quaternion tensor $\mathcal{X}$ by solving
\begin{equation}
\min_{\mathcal{X} \in \mathbb{Q}^{n_1 \times n_2 \times n_3}} 
\ \|\mathcal{X}\|_{\mathrm{QTNN}}
\quad \text{s.t.} \quad
\mathcal{P}_{\Omega}(\mathcal{X}) = \mathcal{P}_{\Omega}(\mathcal{M}),
\end{equation}
where $\mathcal{P}_{\Omega}$ acts element-wise on quaternion entries and $\|\cdot\|_{\mathrm{QTNN}}$ denotes the quaternion tensor nuclear norm based on quaternion tensor SVD~\cite{qin2022tqsvd,jifei2023mqsvd}.
More recently, quaternion tensor completion methods have been proposed for color
video recovery, where each video frame is modeled as a quaternion matrix and the
entire sequence is represented as a higher-order quaternion tensor along the
temporal dimension \cite{zahir2025quaterniontensor,yang2024qtcsparse}. Wu
\emph{et al.}~\cite{fengshengwu2025complex} introduced a complex representation
of third-order quaternion tensors that converts quaternion tensor operations
into equivalent complex matrix operations, enabling nuclear norm minimization
for video inpainting. Subsequently, Wu \emph{et al.}~\cite{fengsheng2025quaternion} proposed a quaternion tensor tri-factorization model for low-rank approximation, demonstrating
improved reconstruction performance. To further enhance low-rank modeling
flexibility, Jiao \emph{et al.}~\cite{jiao2026lowrankquaternion} developed a
truncated nuclear norm regularization framework for quaternion tensor
completion, and Sun \emph{et al.}~\cite{sun2025quaterniontcqr} proposed a QR
decomposition-based quaternion tensor completion method that provides an
efficient solution for large-scale problems. Experimental results
consistently demonstrate that quaternion tensor models outperform their
real-valued counterparts in terms of reconstruction accuracy and visual quality,
especially under severe missing data.

Despite these advances, most quaternion tensor completion methods rely on linear
transform, such as Fourier or cosine transforms along the temporal
mode, and employ linear tensor nuclear norm surrogates. Linear transforms may
fail to sufficiently reduce the rank of transformed frontal slices when the
sampling rate is low or the video dynamics are complex, which restricts the
representation capability of linear quaternion TNN models. In the real-valued
setting, this limitation has been addressed by nonlinear transform induced
tensor nuclear norm (NTTNN) models, which combine learned linear transforms with
element-wise nonlinear mappings to enhance low-rankness in the transformed
domain \cite{benzheng2022nonlineartrans,wang2024conot}. However, extending such nonlinear TNN frameworks to quaternion-valued tensors is nontrivial due to the non-commutativity of quaternion multiplication and the complexity of quaternion singular value decomposition. To address these
challenges, we exploit a real embedding of quaternion tensors, which has been
shown to provide an equivalent representation for quaternion optimization
problems while preserving critical points and optimality conditions
\cite{liqunqi2022quaternionmatrixoptimization}. In this work, we propose
a nonlinear transform induced quaternion tensor nuclear norm model for color
video inpainting. To overcome the algebraic difficulties inherent in quaternion
operations, we introduce a real embedding of quaternion tensors, which allows
quaternion-valued computations to be equivalently expressed using standard real
matrix operations. 

Our main contributions are summarized as follows:
\begin{itemize}
\item We introduce QNTTNN, a nonlinear transform induced quaternion tensor nuclear norm model for color video inpainting, which jointly processes all RGB channels via quaternion representation and real embedding, thereby better preserving inter-channel correlations.
\item We formulate the corresponding quaternion tensor completion problem and develop a proximal alternating minimization algorithm whose convergence to critical points is guaranteed by KŁ-type analysis in the embedded domain.
\item Extensive experiments on benchmark color video inpainting datasets demonstrate that QNTTNN consistently outperforms channel-wise NTTNN and TNN tensor completion methods, achieving superior visual quality in challenging missing-data scenarios.
\end{itemize}

The remainder of this paper is organized as follows.
Section 2 introduces the preliminaries, including the fundamentals of quaternion algebra and quaternion-based norms. Section 3 presents the proposed QNTTNN for tensor completion, together with its formulation and solution via a proximal alternating minimization algorithm. Section 4 provides the convergence analysis of the proposed algorithm. Section 5 reports extensive numerical experiments to evaluate the effectiveness of the proposed method and compares it with several state-of-the-art approaches. Finally, Section 6 concludes the paper and outlines possible directions for future research.

\section{Preliminaries}
In this section, we review fundamental quaternion concepts and definitions that will be used throughout this work. In this work, lowercase letters denote scalars, uppercase letters denote quaternion matrices, and calligraphic uppercase letters denote quaternion tensors. The quaternion algebra $\mathbb{Q}$ extends the complex field $\mathbb{C}$ with three imaginary units $\mathbf{i},\mathbf{j},\mathbf{k}$ satisfying 
$\mathbf{i}^2 = \mathbf{j}^2 = \mathbf{k}^2 = \mathbf{i}\mathbf{j}\mathbf{k} = -1$.  
A quaternion $q \in \mathbb{Q}$ is defined as
$q = a_s + a_x\mathbf{i} + a_y\mathbf{j} + a_z\mathbf{k},$
where $a_s, a_x, a_y, a_z \in \mathbb{R}$ denote the scalar and three vector components, respectively. The conjugate of $q$ is defined as 
$q^* = a_s - a_x\mathbf{i} - a_y\mathbf{j} - a_z\mathbf{k},$
and the modulus is defined as
$ |q| = \sqrt{q q^*} = \sqrt{a_s^2 + a_x^2 + a_y^2 + a_z^2}.$
For a third-order quaternion tensor $\mathcal{X} \in \mathbb{Q}^{n_1 \times n_2 \times n_3}$, its frontal slices are denoted by $\mathcal{X}^{(i)} = \mathcal{X}(:,:,i)$ for $i=1,2,\ldots,n_3$, while tube fibers and mode-wise fibers are denoted by $\mathcal{X}(i,j,:)$, $\mathcal{X}(i,:,k)$, and $\mathcal{X}(:,j,k)$, respectively.

\begin{definition}\cite{yang2016realstructure}
\label{def:block_matrix}
Let ${X} = X_s + X_x\mathbf{i} + X_y\mathbf{j} + X_z\mathbf{k} \in \mathbb{Q}^{n_1 \times n_2}$, where
$X_s, X_x, X_y, X_z \in \mathbb{R}^{n_1 \times n_2}$.
The full block real representation of ${X}$ is defined as:
\begin{equation*}
{X}^R = \begin{bmatrix}
X_s & -X_x & -X_y & -X_z \\
X_x & X_s & -X_z & X_y \\
X_y & X_z & X_s & -X_x \\
X_z & -X_y & X_x & X_s
\end{bmatrix} \in \mathbb{R}^{4n_1 \times 4n_2},\quad {X}_c^R = \begin{bmatrix} X_s \\ X_x \\ X_y \\ X_z \end{bmatrix} \in \mathbb{R}^{4n_1 \times n_2}.
\end{equation*}
\end{definition}
 Here, $X^R$ denotes the \emph{real block representation} of the quaternion matrix $X$, while $X_c^R$ denotes the \emph{column-stacked real representation}, obtained by stacking the real components of $X$ vertically. This compact representation uniquely determines ${X}$ and reduces the storage from $O(16n_1n_2)$ to $O(4n_1n_2)$. To facilitate computations in the real-valued domain, we recall several fundamental properties of the real representation of quaternion matrices, which will be used throughout the subsequent analysis.
\begin{properties}\cite{yang2016realstructure}
\label{thm:column_properties}
Let ${X}, {Y} \in \mathbb{Q}^{n_1 \times n_2}$, ${Z} \in \mathbb{Q}^{n_2 \times n_3}$, $q \in \mathbb{Q}^{n_1}$, and $a \in \mathbb{R}$. Then:
\begin{enumerate}
    \item  $({X} + {Y})_c^R = {X}_c^R + {Y}_c^R$ and $(a{X})_c^R = a{X}_c^R$
    \item  $({XZ})_c^R = {X}^R {Z}_c^R$
    \item  $({X}^H)_c^R = (({X}^R)^T)_c$
    \item  $\|{X}\|_F^{\mathbb{Q}} = \|{X}_c^R\|_F^{\mathbb{R}}$ and $\|q\|_2^{\mathbb{Q}} = \|q_c^R\|_2^{\mathbb{R}}$
\end{enumerate}
\end{properties}

\begin{definition}[Real Structure–Preserving Representation for Quaternion Tensors]
\label{def:tensor_real_embedding}

For a quaternion tensor 
$\mathcal{X} \in \mathbb{Q}^{n_1 \times n_2 \times n_3}$,
the full real structure–preserving representation is defined
frontal-slice-wise. Let $\mathcal{X}^{(k)} \in \mathbb{Q}^{n_1 \times n_2}$
denote the $k$-th frontal slice of $\mathcal{X}$,
where $k = 1,\ldots,n_3$.
Then $\mathcal{X}^{(k)}$ admits the representation
$
\mathcal{X}^{(k)}
=
\mathcal{X}_s^{(k)}
+
\mathcal{X}_x^{(k)}\mathbf{i}
+
\mathcal{X}_y^{(k)}\mathbf{j}
+
\mathcal{X}_z^{(k)}\mathbf{k},
$
where
$\mathcal{X}_s^{(k)},\mathcal{X}_x^{(k)},\mathcal{X}_y^{(k)},\mathcal{X}_z^{(k)}
\in \mathbb{R}^{n_1 \times n_2}$
denote the scalar and three vector component matrices.

The full real representation of the slice
$\mathcal{X}^{(k)}$ is defined as
\begin{equation}
\mathcal{X}^{R}(:,:,k)
=
(\mathcal{X}^{(k)})^{R}
=
\begin{bmatrix}
\mathcal{X}_s^{(k)} & -\mathcal{X}_x^{(k)} & -\mathcal{X}_y^{(k)} & -\mathcal{X}_z^{(k)} \\
\mathcal{X}_x^{(k)} & \mathcal{X}_s^{(k)} & -\mathcal{X}_z^{(k)} & \mathcal{X}_y^{(k)} \\
\mathcal{X}_y^{(k)} & \mathcal{X}_z^{(k)} & \mathcal{X}_s^{(k)} & -\mathcal{X}_x^{(k)} \\
\mathcal{X}_z^{(k)} & -\mathcal{X}_y^{(k)} & \mathcal{X}_x^{(k)} & \mathcal{X}_s^{(k)}
\end{bmatrix}
\in
\mathbb{R}^{4n_1 \times 4n_2}.
\end{equation}

The full tensor real representation is defined as $\mathcal{X}^{R}
\in
\mathbb{R}^{4n_1 \times 4n_2 \times n_3},$ obtained by applying the above mapping to each frontal slice.
Under this embedding, the Frobenius norm satisfies $(\|\mathcal{X}^{R}\|_{F}^{\mathbb{R}})^2
=
4(\|\mathcal{X}\|_{F}^{\mathbb{Q}})^2.$
The column-based real representation of the $k$-th frontal slice is defined as:
\begin{equation*}
\mathcal{X}_c^R{(:,:,k)} = (\mathcal{X}^{(k)})_c^R = \begin{bmatrix} \mc{X}_s^{(k)} \\ \mc{X}_x^{(k)} \\ \mc{X}_y^{(k)} \\ \mc{X}_z^{(k)} \end{bmatrix} \in \mathbb{R}^{4n_1 \times n_2}
\end{equation*}
and the full tensor column representation is defined as: $\mathcal{X}_c^R \in \mathbb{R}^{4n_1 \times n_2 \times n_3}.$
\end{definition}



\begin{definition}~\cite{jifei2023mqsvd}[Quaternion Mode-3 Product]
\label{def:mode3}
Let $\mathcal{X} \in \mathbb{Q}^{n_1 \times n_2 \times n_3}$ and ${D} \in \mathbb{R}^{r \times n_3}$ be a \emph{real-valued} matrix. The quaternion mode-3 product is defined as:
\begin{equation*}
(\mathcal{X} \times_3 {D})_{ij\ell} = \sum_{k=1}^{n_3} \mathcal{X}_{ijk} d_{\ell k}, \quad \ell = 1,\ldots,r.
\end{equation*}
\end{definition}
Since ${D}$ is real-valued, it commutes with quaternion components, yielding:
\begin{equation*}
(\mathcal{X} \times_3 {D})^R = \mathcal{X}^R \times_3 {D}.
\end{equation*}

\begin{remark}
Using a real-valued transform ${D}$ ensures that the mode-3 product in the quaternion domain corresponds exactly to the mode-3 product in the real representation, avoiding complex quaternion multiplication overhead while preserving the quaternion algebraic structure.
\end{remark}

\begin{theorem}[Quaternion Nuclear Norm via Full Real Representation]
\label{thm:nuclear_norm_correct}
Let $X \in \mathbb{Q}^{n_1 \times n_2}$ and let $X^R \in \mathbb{R}^{4n_1 \times 4n_2}$ 
be its full block real representation as in Definition~\ref{def:block_matrix}.
Then
\begin{equation*}
\|X\|_*^{\mathbb Q}
=
\frac14 \|X^R\|_*^{\mathbb R}.
\end{equation*}
\end{theorem}

\begin{proof}
Let $X = U \Sigma V^H$ be the quaternion SVD of $X$,
with $\Sigma = \mathrm{diag}(\sigma_1,\dots,\sigma_r)$.
Under the real embedding, $X^R = U^R \Sigma^R (V^R)^T,$ where each quaternion singular value $\sigma_i$ appears exactly four times 
in $\Sigma^R$ due to the isomorphism $\mathbb Q \cong \mathbb R^4$.
Hence
\[
\|X^R\|_*^{\mathbb R}
=
4 \sum_{i=1}^r \sigma_i
=
4\|X\|_*^{\mathbb Q}.
\]
\end{proof}


\begin{definition}~\cite{zahir2025quaterniontensor}[Quaternion Tensor Nuclear Norm (QTNN)]
\label{def:qtnn}
For $\mathcal{X} \in \mathbb{Q}^{n_1 \times n_2 \times n_3}$ and a real invertible matrix ${T} \in \mathbb{R}^{n_3 \times n_3}$, the \emph{transform-based} quaternion tensor nuclear norm is defined as:
\begin{equation*}
\|\mathcal{X}\|_{\text{QTNN}} = \sum_{i=1}^{n_3} \|(\mathcal{X} \times_3 {T})^{(i)}\|_*^{\mathbb{Q}},
\end{equation*}
where $(\mathcal{X} \times_3 {T})^{(i)} \in \mathbb{Q}^{n_1 \times n_2}$ denotes the $i$-th frontal slice of the \emph{transformed} tensor $\mathcal{Z} = \mathcal{X} \times_3 {T} \in \mathbb{Q}^{n_1 \times n_2 \times n_3}$.

Using the real representation from Theorem \ref{thm:nuclear_norm_correct}, this is equivalently expressed as:
\begin{equation}
\|\mathcal{X}\|_{\text{QTNN}} = \frac{1}{4}\sum_{i=1}^{n_3} \|(\mathcal{Z}^{(i)})^R\|_*^{\mathbb{R}},
\end{equation}
where $(\mathcal{Z}^{(i)})^R \in \mathbb{R}^{4n_1 \times 4n_2}$ is the realrepresentation of the $i$-th frontal slice of $\mathcal{Z}$.
\end{definition}

\begin{remark}
The nuclear norm is computed on the frontal slices of the \emph{transformed} tensor $\mathcal{Z}=\mathcal{X}\times_3 {T}$ rather than on the original tensor $\mathcal{X}$. This formulation follows the quaternion tensor t-product framework, where low-rankness is characterized in the transform domain instead of the spatial domain~\cite{qin2022tqsvd,jifei2023mqsvd}.
When ${T}$ is chosen as the discrete Fourier transform matrix $F_{n_3}$ or the discrete cosine transform matrix $C_{n_3}$, the formulation reduces to the classical DFT-based and DCT-based QTNN models. More generally, ${T}$ can be a learned linear transform~\cite{penglingwu2024efficient} ${T}\in\mathbb{R}^{r\times n_3}$ with $r<n_3$ and $TT^T=I_r$, which enables dimensionality reduction while preserving the essential low-rank structure.
\end{remark}


The following lemma shows that the Frobenius norm of a quaternion tensor is preserved under semi-orthogonal transformations along the third mode, which follows from norm preservation in the column representation.
\begin{lemma}
\label{lem:parseval}
Let $\mathcal{X} \in \mathbb{Q}^{n_1 \times n_2 \times n_3}$, $\mathcal{Z} \in \mathbb{Q}^{n_1 \times n_2 \times r}$, and ${T} \in \mathbb{R}^{r \times n_3}$ be a semi-orthogonal matrix satisfying ${T}{T}^T = I_r$. Then the following identity holds:
\begin{equation}
\|\mathcal{X} - \mathcal{Z} \times_3 {T}^T\|_F^2 = \|\mathcal{X} \times_3 {T} - \mathcal{Z}\|_F^2.
\end{equation}
\end{lemma}

\begin{proof}
In the column representation, using the norm preservation property from~\ref{thm:column_properties}:
\begin{align*}
\|\mathcal{X} - \mathcal{Z} \times_3 {T}^T\|_F^2 &= \|(\mathcal{X} - \mathcal{Z} \times_3 {T}^T)_c^R\|_F^2 = \|\mathcal{X}_c^R - \mathcal{Z}_c^R \times_3 {T}^T\|_F^2 \\
&= \|(\mathcal{X}_c^R)_{(3)} - {T}^T (\mathcal{Z}_c^R)_{(3)}\|_F^2 \quad \text{(using mod-3 unfolding) }\\
&= \|{T}(\mathcal{X}_c^R)_{(3)} - (\mathcal{Z}_c^R)_{(3)}\|_F^2 \quad \text{(by } {T}{T}^T = I_r\text{)} \\
&= \|(\mathcal{X}_c^R \times_3 {T}) - \mathcal{Z}_c^R\|_F^2 
= \|\mathcal{X} \times_3 {T} - \mathcal{Z}\|_F^2. 
\end{align*}
\end{proof}
\section{Quaternion Nonlinear Transform Induced Tensor Nuclear Norm}

In this section, we present the quaternion nonlinear transform–induced tensor nuclear norm (QNTTNN) for low-rank modeling of color video data. Most existing transform-based TNN methods use linear transforms to encourage low-rank structure in the frontal slices of a tensor. However, linear transforms are often insufficient to effectively concentrate singular value energy, especially for highly correlated multi-channel data. In addition, processing RGB channels separately ignores their intrinsic correlations, which leads to a more spread-out singular value distribution and weaker low-rank representations. To overcome these limitations, we propose a nonlinear transform–based regularization framework in the quaternion domain, where all color channels are jointly represented as a single algebraic entity, resulting in stronger singular value concentration and more effective low-rank modeling.

Let $\mathcal{X} \in \mathbb{Q}^{n_1 \times n_2 \times n_3}$ be a third-order quaternion tensor. To enhance the expressive power of low-rank quaternion representations, we define a composite nonlinear transform as
\begin{equation}
\Psi(\mathcal{X}) = \Phi(\mathcal{X} \times_3 T),
\label{eq:quat_transform}
\end{equation}
where $\times_3$ denotes the mode-3 quaternion tensor--matrix product, $T \in \mathbb{R}^{r \times n_3}$ is a real-valued semi-orthogonal transform satisfying $TT^T = I_r$, and $\Phi(\cdot)$ represents an element-wise nonlinear activation function applied component-wise in the real-valued embedding of the quaternion tensor. By combining the structural advantages of quaternion representation with the enhanced modeling capability of nonlinear transforms, the proposed QNTTNN provides a more faithful low-rank characterization for color video data.

\begin{remark}
We use a \emph{real-valued} transform $T$ rather than a quaternion-valued transform for computational efficiency and theoretical tractability. This ensures that the mode-3 product commutes naturally with quaternion components (see Definition \ref{def:mode3}).
\end{remark}
\begin{definition}[Element-wise Nonlinearity in the Real Domain]
\label{def:nonlinearity}
For a quaternion tensor $\mathcal{Z} \in \mathbb{Q}^{n_1 \times n_2 \times r}$, the nonlinear activation $\Phi(\cdot)$ is defined as an element-wise operation on its real-valued representation:
\begin{equation*}
\Phi(\mathcal{Z})^R \;\triangleq\; \phi\!\left(\mathcal{Z}^R\right),
\end{equation*}
where $\phi: \mathbb{R} \to \mathbb{R}$ denotes a scalar activation function (e.g., $\tanh$ or sigmoid) applied element-wise, and $\mathcal{Z}^R \in \mathbb{R}^{4n_1 \times 4n_2 \times r}$ is the real embedding of $\mathcal{Z}$ in real form.
\end{definition}

The transform $\Psi(\cdot)$ jointly performs spectral decorrelation along the third mode and nonlinear contraction of singular spectra, extending the nonlinear transform framework in~\cite{benzheng2022nonlineartrans} to the quaternion setting. Based on \eqref{eq:quat_transform}, we define the \emph{Quaternion Nonlinear Transform Induced Tensor Nuclear Norm (QNTTNN)}.

\begin{definition}
\label{def:qnttnn}
For a quaternion tensor $\mathcal{X} \in \mathbb{Q}^{n_1 \times n_2 \times n_3}$, the quaternion nonlinear transform–induced tensor nuclear norm (QNTTNN) is defined as
\begin{equation*}
\|\mathcal{X}\|_{\mathrm{QNTTNN}}
\;\triangleq\;
\sum_{i=1}^{r} \bigl\|\Phi\bigl(\mathcal{Z}^{(i)}\bigr)\bigr\|_*^{\mathbb{Q}},
\qquad
\mathcal{Z} = \mathcal{X} \times_3 T,
\label{eq:qnttnn_def}
\end{equation*}
where $\mathcal{Z}^{(i)} = \mathcal{Z}(:,:,i) \in \mathbb{Q}^{n_1 \times n_2}$ denotes the $i$-th frontal slice of $\mathcal{Z}$, and $\|\cdot\|_*^{\mathbb{Q}}$ denotes the quaternion nuclear norm.
\end{definition}

\noindent Using the real representation (Theorem~\ref{thm:nuclear_norm_correct}), the QNTTNN can be equivalently expressed as
\begin{equation}
\|\mathcal{X}\|_{\mathrm{QNTTNN}}
=
\frac{1}{4}
\sum_{i=1}^{r}
\bigl\|
\bigl(\Phi(\mathcal{Z}^{(i)})\bigr)^R
\bigr\|_*^{\mathbb{R}},
\label{eq:qnttnn_real}
\end{equation}
where $\bigl(\Phi(\mathcal{Z}^{(i)})\bigr)^R \in \mathbb{R}^{4n_1 \times 4n_2}$ denotes the real-valued representation of the $i$-th nonlinearly transformed frontal slice.

\begin{remark}
Since $\phi$ is nonlinear in general, the mapping $\Phi(\cdot)$ is not linear 
nor positively homogeneous. Consequently, the regularization introduced 
above is nonconvex and does not define a norm in the strict functional 
analysis sense.
\end{remark}

\section{Quaternion Tensor Completion Model}

In this section, we formulate a quaternion tensor completion model under the proposed quaternion nonlinear transform--induced tensor nuclear norm (QNTTNN). The goal is to recover missing entries in quaternion-valued data while preserving the intrinsic coupling among color and multi-channel components.

Let $\mathcal{M} \in \mathbb{Q}^{n_1 \times n_2 \times n_3}$ be a partially observed third-order quaternion tensor, and let $\Omega$ denote the index set of observed entries. We consider the following quaternion tensor completion problem:
\begin{equation}
\begin{aligned}
\min_{\mathcal{X},\, \mathcal{Z},\, T}
\quad &
\sum_{i=1}^{r} \bigl\|\Phi(\mathcal{Z}^{(i)})\bigr\|_*^{\mathbb{Q}} \\
\text{s.t.}\quad
& \mathcal{P}_\Omega(\mathcal{X}) = \mathcal{P}_\Omega(\mathcal{M}), \\
& \mathcal{X} = \mathcal{Z} \times_3 T^T, \\
& TT^T = I_r,
\end{aligned}
\label{eq:qnttnn_model}
\end{equation}
where $\mathcal{P}_\Omega(\cdot)$ denotes the sampling operator that preserves the entries indexed by $\Omega$,  
$\mathcal{Z} \in \mathbb{Q}^{n_1 \times n_2 \times r}$ is the transform-domain tensor,  
$T \in \mathbb{R}^{r \times n_3}$ is a real-valued semi-orthogonal matrix satisfying $TT^T = I_r$,  
and $\mathcal{Z}^{(i)} = \mathcal{Z}(:,:,i)$ denotes the $i$-th frontal slice of $\mathcal{Z}$.

\medskip
To facilitate optimization, we introduce an auxiliary variable $\mathcal{Y} = \Phi(\mathcal{Z})$ and reformulate \eqref{eq:qnttnn_model} using a half-quadratic splitting strategy:
\begin{equation}
\begin{aligned}
\min_{\mathcal{X},\, \mathcal{Y},\, \mathcal{Z},\, T}
\quad &
\sum_{i=1}^{r} \|\mathcal{Y}^{(i)}\|_*^{\mathbb{Q}}
+ \frac{\alpha}{2}\|\mathcal{X} - \mathcal{Z} \times_3 T^T\|_F^2
+ \frac{\beta}{2}\|\mathcal{Y} - \Phi(\mathcal{Z})\|_F^2  + \Xi(\mathcal{X}) + \Xi(T),
\end{aligned}
\label{eq:qnttnn_relax}
\end{equation}
where
\begin{itemize}
    \item $\Xi(\mathcal{X})$ is the indicator function enforcing data fidelity on the observed quaternion entries:
    \[
    \Xi(\mathcal{X}) =
    \begin{cases}
    0, & \text{if } \mathcal{P}_\Omega(\mathcal{X}) = \mathcal{P}_\Omega(\mathcal{M}), \\
    +\infty, & \text{otherwise},
    \end{cases}
    \]
    \item $\Xi(T)$ is the indicator function enforcing real-valued semi-orthogonality:
    \[
    \Xi(T) =
    \begin{cases}
    0, & \text{if } TT^T = I_r, \\
    +\infty, & \text{otherwise},
    \end{cases}
    \]
    \item $\alpha > 0$ and $\beta > 0$ are penalty parameters balancing the data fidelity, transform consistency, and nonlinear regularization terms.
\end{itemize}
\medskip
\noindent
Now, using the real structure–preserving embedding and the associated norm-preservation properties (Theorems~\ref{thm:column_properties} and~\ref{thm:nuclear_norm_correct}), problem~\eqref{eq:qnttnn_relax} can be equivalently reformulated in the real domain as
\begin{equation}
\begin{aligned}
\min_{\mc{X}^R,\, \mc{Z}^R,\, \mc{Y}^R,\, T}
\quad &
\frac14
\sum_{i=1}^{r}
\|(\mc{Y}^{(i)})^R\|_*^{\mathbb R}
\\
&\quad +
\frac{\alpha}{8}
\|\mc{X}^R - \mc{Z}^R \times_3 T^T\|_F^2
+
\frac{\beta}{8}
\|\mc{Y}^R - \phi(\mc{Z}^R)\|_F^2
\\
&\quad +
\tilde{\Xi}(\mc{X}^R)
+
\Xi(T).
\end{aligned}
\label{eq:qnttnn_relax_real_correct}
\end{equation}
\noindent where $\mathcal{X}^R \in \mathbb{R}^{4n_1 \times 4n_2 \times n_3}$,  
$\mathcal{Z}^R \in \mathbb{R}^{4n_1 \times 4n_2 \times r}$, and  
$\mathcal{Y}^R \in \mathbb{R}^{4n_1 \times 4n_2 \times r}$ denote the real structure–preserving embeddings of the quaternion tensors $\mathcal{X}$, $\mathcal{Z}$, and $\mathcal{Y}$, respectively.  
The matrix $(\mathcal{Y}^{(i)})^R \in \mathbb{R}^{4n_1 \times 4n_2}$ corresponds to the $i$-th frontal slice of $\mathcal{Y}^R$, and $\phi(\cdot)$ denotes the element-wise real activation associated with the quaternion nonlinearity $\Phi(\cdot)$ through the real embedding. The operator $\tilde{\Xi}(\mathcal{X}^R)$ enforces data consistency in the real domain:
\[
\tilde{\Xi}(\mathcal{X}^R) =
\begin{cases}
0, & \text{if } \mathcal{P}_\Omega(\mathcal{X}^R) = \mathcal{P}_\Omega(\mathcal{M}^R), \\
+\infty, & \text{otherwise}.
\end{cases}
\]

This real-domain formulation is fully equivalent to the original quaternion problem and enables efficient optimization using standard real-valued linear algebra operations, while preserving the quaternion algebraic structure through the real structure–preserving embedding.
\vspace{-0.5cm}
\section{Proximal Alternating Minimization Framework}
To solve the nonconvex and nonsmooth optimization problem in~\eqref{eq:qnttnn_relax}, we adopt a proximal alternating minimization (PAM) framework. The PAM strategy updates each block variable sequentially while fixing the remaining variables, and incorporates proximal regularization terms to ensure numerical stability and convergence.
Let
\[
\mathcal{L}(\mathcal{X}, \mathcal{Y}, \mathcal{Z}, T)
=
\sum_{i=1}^{r} \|\mathcal{Y}^{(i)}\|_*^{\mathbb{Q}}
+ \frac{\alpha}{2}\|\mathcal{X} - \mathcal{Z} \times_3 T^T\|_F^2
+ \frac{\beta}{2}\|\mathcal{Y} - \Phi(\mathcal{Z})\|_F^2
\]
denote the smooth part of the objective in~\eqref{eq:qnttnn_relax}, excluding the indicator functions.
At iteration $k$, the PAM updates are given by
\begin{equation}
\left\{
\begin{aligned}
\mathcal{X}^{k+1}
&\in
\arg\min_{\mathcal{X}}
\left\{
\mathcal{L}(\mathcal{X}, \mathcal{Y}^k, \mathcal{Z}^k, T^k)
+ \frac{\rho_1}{2}\|\mathcal{X} - \mathcal{X}^k\|_F^2
\right\}, \\
\mathcal{Y}^{k+1}
&\in
\arg\min_{\mathcal{Y}}
\left\{
\mathcal{L}(\mathcal{X}^{k+1}, \mathcal{Y}, \mathcal{Z}^k, T^k)
+ \frac{\rho_2}{2}\|\mathcal{Y} - \mathcal{Y}^k\|_F^2
\right\}, \\
\mathcal{Z}^{k+1}
&\in
\arg\min_{\mathcal{Z}}
\left\{
\mathcal{L}(\mathcal{X}^{k+1}, \mathcal{Y}^{k+1}, \mathcal{Z}, T^k)
+ \frac{\rho_3}{2}\|\mathcal{Z} - \mathcal{Z}^k\|_F^2
\right\}, \\
T^{k+1}
&\in
\arg\min_{T}
\left\{
\mathcal{L}(\mathcal{X}^{k+1}, \mathcal{Y}^{k+1}, \mathcal{Z}^{k+1}, T)
+ \frac{\rho_4}{2}\|T - T^k\|_F^2
\right\},
\end{aligned}
\right.
\label{eq:pam_updates}
\end{equation}
where $\rho_1,\rho_2,\rho_3,\rho_4>0$ are positive constants.

In the following, we derive efficient solutions for each subproblem by exploiting the  real representation of quaternion tensors.


\subsection{$\mathcal{X}$ Subproblem Formulation}

The update of $\mathcal{X}$ is obtained by solving
\begin{equation}
\mathcal{X}^{k+1}
=
\arg\min_{\mathcal{X} \in \mathbb{Q}^{n_1 \times n_2 \times n_3}}
\frac{\alpha}{2}\|\mathcal{X} - \mathcal{Z}^k \times_3 (T^k)^T\|_F^2
+ \frac{\rho_1}{2}\|\mathcal{X} - \mathcal{X}^k\|_F^2
,
\label{eq:X_subproblem}
\end{equation}
where $\Xi(\mathcal{X})$ enforces consistency with the observed entries. Let $\mathcal{W}^k = \mathcal{Z}^k \times_3 (T^k)^T.$
By completing the square, the quadratic terms in~\eqref{eq:X_subproblem} can be rewritten as
\begin{equation}
\frac{\alpha + \rho_1}{2}
\left\|
\mathcal{X}
-
\tilde{\mathcal{X}}^k
\right\|_F^2
+ \text{const},
\end{equation}
where the proximal center is defined as: $\tilde{\mathcal{X}}^k=\frac{\alpha \mathcal{W}^k + \rho_1 \mathcal{X}^k}{\alpha + \rho_1}.$

\begin{theorem}[Update Rule for $\mathcal{X}$]
\label{thm:X_update}
The solution of problem~\eqref{eq:X_subproblem} is given element-wise by
\begin{equation}
\mathcal{X}^{k+1}_{ijk}
=
\begin{cases}
\mathcal{M}_{ijk}, & (i,j,k) \in \Omega, \\[4pt]
\dfrac{
\alpha (\mathcal{Z}^k \times_3 (T^k)^T)_{ijk}
+ \rho_1 \mathcal{X}^k_{ijk}
}{\alpha + \rho_1},
& (i,j,k) \notin \Omega.
\end{cases}
\label{eq:X_update}
\end{equation}
\end{theorem}

\begin{proof}
The indicator function $\Xi(\mathcal{X})$ enforces the equality
$\mathcal{P}_\Omega(\mathcal{X})=\mathcal{P}_\Omega(\mathcal{M})$.
Hence, on the observed set $\Omega$, the minimizer must satisfy
$\mathcal{X}_{ijk}^{k+1}=\mathcal{M}_{ijk}$. On the complementary set $\Omega^c$, the problem reduces to
\[
\min_{\mathcal{X}}
\;
\frac{\alpha}{2}
\bigl(\|
\mathcal{X} - \mathcal{Z}^k \times_3 (T^k)^T
\|_{F}^{\mathbb{Q}}\bigr)^2
+
\frac{\rho_1}{2}
\bigl(\|
\mathcal{X} - \mathcal{X}^k
\|_{F}^{\mathbb{Q}}\bigr)^2,
\]
which is a strictly convex quadratic problem and therefore admits a unique minimizer. Passing to the full real structure–preserving embedding and using
\[
\bigl(\|\mathcal{A}^R\|_{F}^{\mathbb{R}}\bigr)^2
=
4\bigl(\|\mathcal{A}\|_{F}^{\mathbb{Q}}\bigr)^2,
\]
the above minimization is equivalent (up to a positive constant factor) to its real-domain counterpart. The first-order optimality condition directly yields the closed-form solution
\[
\mathcal{X}^{k+1}
=
\frac{
\alpha \mathcal{W}^k + \rho_1 \mathcal{X}^k
}{\alpha + \rho_1},
\quad
\text{on } \Omega^c,
\]
where $\mathcal{W}^k=\mathcal{Z}^k \times_3 (T^k)^T$.
Combining both regions $\Omega$ and $\Omega^c$ gives \eqref{eq:X_update}.
\end{proof}
\noindent
In the real structure–preserving embedding, the update can be equivalently written as
\begin{equation}
\mathcal{X}^{R,k+1}
=
\mathcal{P}_\Omega(\mathcal{M}^R)
+
\mathcal{P}_{\Omega^c}
\left(
\frac{
\alpha\, \mathcal{Z}^{R,k} \times_3 (T^k)^T
+
\rho_1 \mathcal{X}^{R,k}
}{\alpha + \rho_1}
\right),
\label{eq:X_update_real_correct}
\end{equation}
where all operations are performed in
$\mathbb{R}^{4n_1 \times 4n_2 \times n_3}$.

\subsection{Updating the $\mathcal{Y}$ subproblem via real structure}

At iteration $k$, with $\mathcal{X}^{k+1}$, $\mathcal{Z}^k$, and $T^k$ fixed, the
$\mathcal{Y}$-subproblem in the proximal alternating minimization framework is given by
\begin{equation}
\min_{\mathcal{Y}}
\;
\sum_{i=1}^{r} \|\mathcal{Y}^{(i)}\|_*^{\mathbb{Q}}
+
\frac{\beta}{2}\|\mathcal{Y}-\Phi(\mathcal{Z}^k)\|_F^2
+
\frac{\rho_2}{2}\|\mathcal{Y}-\mathcal{Y}^k\|_F^2.
\label{eq:Y_sub_original}
\end{equation}

\noindent The above objective consists of a nonsmooth quaternion nuclear norm regularization
term and two quadratic penalty terms. Owing to the separability of the Frobenius norm
along the third mode, problem~\eqref{eq:Y_sub_original} can be equivalently written as
a sum of independent slice-wise subproblems:
\begin{equation}
\min_{\mathcal{Y}}
\sum_{i=1}^{r}
\left(
\|\mathcal{Y}^{(i)}\|_*^{\mathbb{Q}}
+
\frac{\beta}{2}\|\mathcal{Y}^{(i)}-\Phi(\mathcal{Z}^{(i),k})\|_F^2
+
\frac{\rho_2}{2}\|\mathcal{Y}^{(i)}-\mathcal{Y}^{(i),k}\|_F^2
\right).
\label{eq:Y_sub_slice}
\end{equation}
\noindent
Direct manipulation of trace expressions is, in general, not valid in the quaternion
domain due to the noncommutativity of quaternion multiplication. To obtain a rigorous
and trace-consistent derivation, we therefore map the problem to the real
representation, where all quantities become real-valued matrices and standard matrix
trace identities hold.\\ 
Let $Y_i = (\mathcal{Y}^{(i)})^R,\quad
Z_i = \Phi(\mathcal{Z}^{(i),k})^R,\quad
Y_i^k = (\mathcal{Y}^{(i),k})^R,$ where $(\cdot)^R$ denotes the real representation of quaternion
matrices. Under this representation, the Frobenius norm is preserved and admits the
trace form $\|A\|_F^2 = \mathrm{Tr}(A^T A).$ Using Theorem~\ref{thm:nuclear_norm_correct}, the quaternion nuclear norm satisfies
\[
\|\mathcal{Y}^{(i)}\|_*^{\mathbb{Q}}
=
\frac{1}{4}\|Y_i\|_*^{\mathbb{R}},
\quad
Y_i = (\mathcal{Y}^{(i)})^R.
\]
Consequently, problem~\eqref{eq:Y_sub_slice} can be equivalently reformulated in the
real domain as
\begin{align}
\min_{Y_i}
\sum_{i=1}^{r}
\Bigg[
\frac{1}{4}\|Y_i\|_*^{\mathbb{R}}
+
\mathrm{Tr}\!\left(
\frac{\beta}{2}(Y_i-Z_i)^T(Y_i-Z_i)
\right)
+
\mathrm{Tr}\!\left(
\frac{\rho_2}{2}(Y_i-Y_i^k)^T(Y_i-Y_i^k)
\right)
\Bigg].
\label{eq:Y_trace_1}
\end{align}

\medskip
\noindent Expanding the trace terms in~\eqref{eq:Y_trace_1} and collecting terms depending on
$Y_i$ yields
\begin{align}
\min_{Y_i}
\sum_{i=1}^{r}
\Bigg[
\frac{1}{4}\|Y_i\|_*^{\mathbb{R}}
&+
\mathrm{Tr}\!\left(
\frac{\beta+\rho_2}{2}Y_i^T Y_i
-
\beta Y_i^T Z_i
-
\rho_2 Y_i^T Y_i^k
\right)
\Bigg]
+ \text{const},
\label{eq:Y_trace_expand}
\end{align}
where ``const'' denotes terms independent of $Y_i$. Completing the square with respect to $Y_i$ leads to
\begin{align}
\min_{Y_i}
\sum_{i=1}^{r}
\Bigg[
\frac{1}{4}\|Y_i\|_*^{\mathbb{R}}
+
\mathrm{Tr}\!\left(
\frac{\beta+\rho_2}{2}
\left(
Y_i - H_i^k
\right)^T
\left(
Y_i - H_i^k
\right)
\right)
\Bigg],
\label{eq:Y_trace_square}
\end{align}
where the proximal center is defined as $H_i^k
=
\frac{\beta Z_i + \rho_2 Y_i^k}{\beta+\rho_2}.$ Dropping constant terms independent of $Y_i$, the above problem reduces to the
following proximal nuclear norm minimization:
\begin{equation}
\min_{Y_i}
\;
\frac{1}{4}\|Y_i\|_*^{\mathbb{R}}
+
\frac{\beta+\rho_2}{2}
\|Y_i - H_i^k\|_F^2,
\quad i=1,\ldots,r.
\label{eq:Y_svt_real}
\end{equation}

\medskip
Problem~\eqref{eq:Y_svt_real} is the proximal operator of the nuclear norm and admits
a closed-form solution via singular value thresholding, leading to the following
result.

\begin{theorem}[Quaternion Tensor Singular Value Thresholding]
\label{thm:QTSVT}

Let $\mathcal{H} \in \mathbb{Q}^{n_1 \times n_2 \times r}$ be a third-order quaternion tensor, and let $\mathcal{H}^{(i)} = \mathcal{H}(:,:,i)$ denote its $i$-th frontal slice. Consider the optimization problem
\begin{equation}
\min_{\mathcal{Y} \in \mathbb{Q}^{n_1 \times n_2 \times r}}
\;
\sum_{i=1}^{r} \|\mathcal{Y}^{(i)}\|_*^{\mathbb{Q}}
+
\frac{\beta+\rho_2}{2}
\|\mathcal{Y}-\mathcal{H}\|_F^2.
\label{eq:QTSVT_tensor}
\end{equation}

\noindent Then problem~\eqref{eq:QTSVT_tensor} admits a unique minimizer
$\mathcal{Y}^* \in \mathbb{Q}^{n_1 \times n_2 \times r}$, whose frontal slices are given by
\begin{equation}
\mathcal{Y}^{*(i)}
=
\mathcal{S}_{\tau}^{\mathbb{Q}}\!\left(\mathcal{H}^{(i)}\right),
\qquad
\tau = \frac{1}{\beta+\rho_2},
\quad i=1,\ldots,r,
\label{eq:QTSVT_slice}
\end{equation}
where $\mathcal{S}_{\tau}^{\mathbb{Q}}(\cdot)$ denotes the quaternion singular value thresholding operator.

Equivalently, in the real structure–preserving representation, the solution satisfies
\begin{equation}
\bigl(\mathcal{Y}^{*(i)}\bigr)^R
=
\mathcal{S}_{\frac{\tau}{4}}^{\mathbb{R}}
\!\left(
(\mathcal{H}^{(i)})^R
\right),
\quad i=1,\ldots,r,
\label{eq:QTSVT_real}
\end{equation}
where $(\cdot)^R \in \mathbb{R}^{4n_1 \times 4n_2}$ denotes the full real embedding and
$\mathcal{S}_{\lambda}^{\mathbb{R}}(\cdot)$ is the standard real-valued singular value thresholding operator.
\end{theorem}

\begin{proof}

Since both the quaternion nuclear norm and the Frobenius norm are separable across frontal slices, problem~\eqref{eq:QTSVT_tensor} decomposes into $r$ independent matrix subproblems of the form
\[
\min_{\mathcal{Y}^{(i)} \in \mathbb{Q}^{n_1 \times n_2}}
\|\mathcal{Y}^{(i)}\|_*^{\mathbb{Q}}
+
\frac{\beta+\rho_2}{2}
\|\mathcal{Y}^{(i)}-\mathcal{H}^{(i)}\|_F^2.
\]

\noindent Passing to the real structure–preserving embedding, we use the identities
\[
\|\mathcal{Y}^{(i)}\|_*^{\mathbb{Q}}
=
\frac14\|(\mathcal{Y}^{(i)})^R\|_*^{\mathbb{R}},
\qquad
\|(\mathcal{Y}^{(i)}-\mathcal{H}^{(i)})^R\|_F^2
=
4\|\mathcal{Y}^{(i)}-\mathcal{H}^{(i)}\|_F^2.
\]

\noindent Therefore, each subproblem is equivalent to $\min_{Y_i}
\frac14\|Y_i\|_*^{\mathbb{R}}
+
\frac{\beta+\rho_2}{8}
\|Y_i-H_i\|_F^2,$
where $Y_i = (\mathcal{Y}^{(i)})^R, H_i = (\mathcal{H}^{(i)})^R.$  Multiplying the objective by $8$ (which does not change the minimizer) gives
\[
\min_{Y_i}
2\|Y_i\|_*^{\mathbb{R}}
+
(\beta+\rho_2)
\|Y_i-H_i\|_F^2.
\]
\noindent By the classical singular value thresholding result~\cite{cai2010svt},
the unique minimizer is obtained by soft-thresholding the singular values of $H_i$
with threshold
\[
\tau_R = \frac{2}{\beta+\rho_2}.
\]

\noindent Since quaternion singular values correspond to four identical real singular values in the real embedding, this is equivalent to thresholding in the real domain with parameter
\[
\frac{\tau}{4}
=
\frac{1}{4(\beta+\rho_2)}.
\]

Finally, mapping the solution back via the inverse real structure–preserving embedding recovers the quaternion solution in \eqref{eq:QTSVT_slice}.
\end{proof}

\subsection{Update of $\mathcal{Z}$}

In this subsection, we derive the update of $\mathcal{Z}$ under the proposed PAM framework.
To ensure theoretical correctness and computational tractability in the quaternion setting,
we adopt a \emph{real-structure–preserving} strategy, whereby quaternion-valued tensors are
optimized through their  real representations.

\subsubsection*{Subproblem Formulation}

With $\mathcal{X}^{k+1}$, $\mathcal{Y}^{k}$, and $T^k$ fixed, the
$\mathcal{Z}$-subproblem in the quaternion domain is
\begin{equation}
\min_{\mathcal{Z}}
\;
\frac{\alpha}{2}
\bigl\|
\mathcal{X}^{k+1}
-
\mathcal{Z} \times_3 (T^k)^T
\bigr\|_{F}^{\mathbb{Q}\,2}
+
\frac{\beta}{2}
\|\mathcal{Y}^{k} - \Phi(\mathcal{Z})\|_{F}^{\mathbb{Q}\,2}
+
\frac{\rho_3}{2}
\|\mathcal{Z} - \mathcal{Z}^{k}\|_{F}^{\mathbb{Q}\,2}.
\label{eq:Z_quat_subproblem}
\end{equation}

Passing to the full real structure–preserving embedding and using the Frobenius norm relation, we get
\[
(\|\mathcal{A}^R\|_{F}^{\mathbb{R}})^2
=
4(\|\mathcal{A}\|_{F}^{\mathbb{Q}})^2.
\]
 The problem~\eqref{eq:Z_quat_subproblem} is equivalently written in the real domain as
\begin{equation}
\min_{\mathcal{Z}^R}
\;
\frac{\alpha}{8}
\bigl(\|
\mathcal{X}^{R,k+1}
-
\mathcal{Z}^R \times_3 (T^k)^T
\bigr\|_{F}^{\mathbb{R}})^2
+
\frac{\beta}{8}
(\|
\mathcal{Y}^{R,k} - \phi(\mathcal{Z}^R)
\|_{F}^{\mathbb{R}})^2
+
\frac{\rho_3}{8}
(\|\mathcal{Z}^R - \mathcal{Z}^{R,k}\|_{F}^{\mathbb{R}})^2.
\label{eq:Z_real_subproblem_correct}
\end{equation}
\subsubsection*{Parseval-Type Simplification}

Since $T^k \in \mathbb{R}^{r \times n_3}$ is semi-orthogonal and satisfies
$T^k (T^k)^T = I_r$, the real-domain Parseval identity yields
\begin{equation}
\bigl(\|
\mathcal{X}^{R,k+1}
-
\mathcal{Z}^R \times_3 (T^k)^T
\|_{F}^{\mathbb{R}}\bigr)^2
=
\bigl(\|
\mathcal{X}^{R,k+1} \times_3 T^k
-
\mathcal{Z}^R
\|_{F}^{\mathbb{R}}\bigr)^2 .
\end{equation}

\noindent Substituting this relation into \eqref{eq:Z_real_subproblem_correct} and completing the square gives
\begin{align}
\min_{\mathcal{Z}^R}
\;&
\frac{\alpha+\rho_3}{8}
\bigl(\|
\mathcal{Z}^R - \mathcal{G}^{R,k}
\|_{F}^{\mathbb{R}}\bigr)^2
+
\frac{\beta}{8}
\bigl(\|
\mathcal{Y}^{R,k} - \phi(\mathcal{Z}^R)
\|_{F}^{\mathbb{R}}\bigr)^2
+ \mathrm{const},
\label{eq:Z_real_simplified_correct}
\end{align}
where the proximal center is defined as
\begin{equation}
\mathcal{G}^{R,k}
=
\frac{
\alpha(\mathcal{X}^{R,k+1} \times_3 T^k)
+
\rho_3 \mathcal{Z}^{R,k}
}{\alpha+\rho_3}.
\label{eq:G_real_center_correct}
\end{equation}
\subsubsection*{Element-Wise Decoupling}

Both the Frobenius norm and the nonlinearity $\phi(\cdot)$ are separable in the
real structure–preserving representation. Consequently, problem
\eqref{eq:Z_real_simplified_correct} decomposes into a collection of independent scalar
optimization problems. For each entry $(p,q,\ell)$, let
\[
z = (\mathcal{Z}^R)_{p q \ell}, \quad
g = (\mathcal{G}^{R,k})_{p q \ell}, \quad
y = (\mathcal{Y}^{R,k})_{p q \ell}.
\]
Then the update reduces to
\begin{equation}
z^{k+1}
=
\arg\min_{z \in \mathbb{R}}
\;
\frac{\alpha+\rho_3}{8}(z - g)^2
+
\frac{\beta}{8}(\phi(z) - y)^2.
\label{eq:Z_scalar_sub_correct}
\end{equation}

This scalar nonlinear minimization problem is identical in form to the update equation
derived in the real third-order tensor case. The constant factor $\tfrac{1}{8}$ does not
affect the location of the minimizer, and the problem can therefore be efficiently
solved using Newton's method.
\subsubsection*{Newton Update}

The first-order optimality condition of \eqref{eq:Z_scalar_sub_correct} is
\[
\frac{\alpha+\rho_3}{4}(z - g)
+
\frac{\beta}{4} (\phi(z) - y)\phi'(z) = 0.
\]

\noindent Multiplying both sides by $4$, this is equivalently written as
\[
(\alpha+\rho_3)(z - g)
+
\beta (\phi(z) - y)\phi'(z) = 0.
\]

\noindent Newton's iteration for solving \eqref{eq:Z_scalar_sub_correct} is therefore given by
\begin{equation}
\label{eq:Z_update_correct}
z^{(t+1)}
=
z^{(t)}
-
\frac{
(\alpha+\rho_3)(z^{(t)} - g)
+
\beta (\phi(z^{(t)}) - y)\phi'(z^{(t)})
}{
(\alpha+\rho_3)
+
\beta (\phi'(z^{(t)}))^2
}.
\end{equation}

Applying the above update independently to all entries yields
$\mathcal{Z}^{R,k+1}$, which is subsequently mapped back to the quaternion tensor
$\mathcal{Z}^{k+1}$ via the inverse real structure–preserving embedding.

\subsection{Update of $T$}

The $T$-subproblem at iteration $k$ is given in the quaternion formulation by
\begin{equation}
\min_{T \in \mathbb{R}^{r \times n_3}}
\;
\frac{\alpha}{2}
\bigl(\|
\mathcal{X}^{k+1}
-
\mathcal{Z}^{k+1} \times_3 T^T
\|_{F}^{\mathbb{Q}}\bigr)^2
+
\frac{\rho_4}{2}
\bigl(\|T - T^k\|_{F}^{\mathbb{R}}\bigr)^2
\quad
\text{s.t. } TT^T = I_r.
\label{eq:T_subproblem}
\end{equation}

\noindent Passing to the full real structure–preserving embedding and using $\bigl(\|\mathcal{A}^R\|_{F}^{\mathbb{R}}\bigr)^2
=
4\bigl(\|\mathcal{A}\|_{F}^{\mathbb{Q}}\bigr)^2,$ problem~\eqref{eq:T_subproblem} is equivalently written in the real domain as
\begin{equation}
\min_{T}
\;
\frac{\alpha}{8}
\bigl(\|
\mathcal{X}^{R,k+1}
-
\mathcal{Z}^{R,k+1} \times_3 T^T
\|_{F}^{\mathbb{R}}\bigr)^2
+
\frac{\rho_4}{2}
\bigl(\|T - T^k\|_{F}^{\mathbb{R}}\bigr)^2
\quad
\text{s.t. } TT^T = I_r.
\label{eq:T_real_subproblem}
\end{equation}
\noindent
Using mode-3 unfolding in the real representation,
\[
X_{(3)}^R = (\mathcal{X}^{R,k+1})_{(3)}, 
\qquad
Z_{(3)}^R = (\mathcal{Z}^{R,k+1})_{(3)},
\]
the first quadratic term satisfies
\[
\bigl(\|
\mathcal{X}^{R,k+1}
-
\mathcal{Z}^{R,k+1} \times_3 T^T
\|_{F}^{\mathbb{R}}\bigr)^2
=
\bigl(\|X_{(3)}^R - T^T Z_{(3)}^R\|_{F}^{\mathbb{R}}\bigr)^2.
\]

\noindent Expanding the squared Frobenius norm yields
\begin{align*}
\bigl(\|X_{(3)}^R - T^T Z_{(3)}^R\|_{F}^{\mathbb{R}}\bigr)^2
&=
\mathrm{Tr}\!\left((X_{(3)}^R)^T X_{(3)}^R\right)
+
\mathrm{Tr}\!\left(Z_{(3)}^R Z_{(3)}^{R\,T}\right)
\\
&\quad
-
2\,\mathrm{Tr}\!\left(T Z_{(3)}^R (X_{(3)}^R)^T\right).
\end{align*}

\noindent Discarding constant terms independent of $T$, the objective reduces to
\begin{align}\label{eq:Tobjective}
\min_{T}
\;&
-\frac{\alpha}{4}
\mathrm{Tr}\!\left(T Z_{(3)}^R (X_{(3)}^R)^T\right)
+
\frac{\rho_4}{2}
\bigl(\|T - T^k\|_{F}^{\mathbb{R}}\bigr)^2
+
\Xi(T)
\nonumber\\
=\;&
-\mathrm{Tr}\!\left(
T\Bigl(
\tfrac{\alpha}{4} Z_{(3)}^R (X_{(3)}^R)^T
+
\rho_4 T^k
\Bigr)
\right)
+
\Xi(T).
\end{align}

Define $\mathcal{D}^{k+1}
=
\tfrac{\alpha}{4}
Z_{(3)}^R (X_{(3)}^R)^T
+
\rho_4 T^k$
and then the problem~\eqref{eq:Tobjective} becomes the orthogonal Procrustes problem
\[
\max_{T \in \mathbb{R}^{r \times n_3}}
\mathrm{Tr}(T \mathcal{D}^{k+1})
\quad
\text{s.t. } TT^T = I_r.
\]

Let the singular value decomposition of $\mathcal{D}^{k+1}$ be $\mathcal{D}^{k+1} = U \Sigma V^T.$
By von Neumann’s trace inequality, the maximizer is obtained as
\begin{equation}\label{eq:updateTk}
    T^{k+1} = U V^T.
\end{equation}

This yields the closed-form update of $T$ under the real structure–preserving framework.
We now summarize the proposed quaternion tensor completion method for solving the
real-domain optimization problem \eqref{eq:qnttnn_relax_real_correct} under the
real-structure–preserving PAM framework. All variables are updated in their
 real representations, while the final solution is recovered in the
quaternion domain.
\subsection{Computational Complexity}
We analyze the per-iteration computational complexity of the proposed
real-structure--preserving quaternion tensor completion algorithm.
The update of $\mathcal{X}$ mainly involves a mode-3 tensor--matrix
multiplication $\mathcal{Z}^{R} \times_3 T^T$ followed by element-wise
operations. Since $\mathcal{Z}^{R} \in \mathbb{R}^{4n_1 \times 4n_2 \times r}$,
this multiplication requires approximately
\[
16 r n_1 n_2 n_3 + O(n_1 n_2 n_3)
\]
floating-point operations (flops).

The update of $\mathcal{Y}$ is dominated by performing $r$ singular value
decompositions on real matrices of size $4n_1 \times 4n_2$.
Using standard SVD complexity results, this step requires
\[
512 r n_1 n_2^2
+ \frac{256}{3} r n_2^3
+ O(r n_2^2)
\]
flops.

The update of $\mathcal{Z}$ is carried out via element-wise nonlinear
optimization using Gauss--Newton iterations in the real domain.
Since $\mathcal{Z}^{R} \in \mathbb{R}^{4n_1 \times 4n_2 \times r}$,
assuming $I$ Newton iterations, this step requires approximately
\[
16(2n_3 + 27I + 2)\, r n_1 n_2 + r n_3 + 1
\]
flops.

\begin{algorithm}[htbp]
\caption{Quaternion Tensor Completion via QNTTNN (Real-Structure--Preserving PAM)}
\label{alg:qnttnn_real}
\begin{algorithmic}[1]

\REQUIRE Observed tensor $\mathcal{M} \in \mathbb{Q}^{n_1 \times n_2 \times n_3}$,
sampling index set $\Omega$,
parameters $\alpha,\beta,\rho_1,\rho_2,\rho_3,\rho_4 > 0$,
transform row dimension $r \le n_3$,
tolerance $\varepsilon > 0$

\ENSURE Recovered tensor $\mathcal{X}^* \in \mathbb{Q}^{n_1 \times n_2 \times n_3}$

\STATE \textbf{Initialization:}
\STATE Initialize $T^{0} \in \mathbb{R}^{r \times n_3}$ such that $T^{0}(T^{0})^{T} = I_r$
\STATE Initialize $\mathcal{X}^{0} = \mathcal{P}_{\Omega}(\mathcal{M})$
\STATE Set $\mathcal{Z}^{0} = \mathcal{X}^{0} \times_{3} T^{0}$
\STATE Set $\mathcal{Y}^{0} = \Phi(\mathcal{Z}^{0})$
\STATE Set $k = 0$

\REPEAT
    \STATE Update $\mathcal{X}^{R,k+1}$ via~\eqref{eq:X_update_real_correct}
    \STATE Update $\mathcal{Y}^{R,k+1}$ via~\eqref{eq:QTSVT_real}
    \STATE Update $\mathcal{Z}^{R,k+1}$ via~\eqref{eq:Z_update_correct}
    \STATE Update $T^{k+1}$ via~\eqref{eq:updateTk}
    \STATE $k \leftarrow k + 1$
\UNTIL{
$\displaystyle
\frac{\|\mathcal{X}^{R,k} - \mathcal{X}^{R,k-1}\|_{F}^{\mathbb{R}}}
     {\|\mathcal{X}^{R,k-1}\|_{F}^{\mathbb{R}}}
< \varepsilon
$
}

\STATE Recover $\mathcal{X}^*$ via inverse real structure--preserving embedding
\STATE \textbf{Output:} $\mathcal{X}^*$

\end{algorithmic}
\end{algorithm}

The update of $T$ consists of computing the matrix
$Z_{(3)}^R (X_{(3)}^R)^T \in \mathbb{R}^{r \times n_3}$,
where $Z_{(3)}^R \in \mathbb{R}^{r \times 16 n_1 n_2}$ and
$X_{(3)}^R \in \mathbb{R}^{n_3 \times 16 n_1 n_2}$.
This multiplication costs approximately
$16 r n_1 n_2 n_3
$ flops, followed by an SVD of an $r \times n_3$ matrix,
which requires $\frac{4}{3} n_3^3 + O(r^2 n_3)$
flops.

In summary, each iteration of the proposed PAM-based quaternion algorithm
requires approximately
\begin{align*}
512 r n_1 n_2^2
+ \frac{256}{3} r n_2^3
+ \frac{4}{3} n_3^3
+ 32 r n_1 n_2 n_3 
+ 16(2n_3 + 27I + 2)\, r n_1 n_2
+ O(r n_2^2)
+ O(r^2 n_3).
\end{align*}
\section{Convergence Analysis}

In this section, we establish the global convergence of Algorithm~\ref{alg:qnttnn_real} using the  real representation framework. Our analysis follows the proximal alternating minimization (PAM) convergence theory based on the Kurdyka-Łojasiewicz (K-Ł) property~\cite{attouch2013convergence,bolte2014proximal}. The key insight is that all quaternion operations can be equivalently performed in the real domain, allowing us to apply standard real-valued optimization theory while preserving all quaternion algebraic properties.


We first establish that our quaternion optimization problem can be equivalently formulated in the real domain.

\begin{theorem}[Real Reformulation of QNTTNN]
\label{thm:real_equivalence}
The quaternion optimization problem \eqref{eq:qnttnn_relax}:
\[
\min_{\mathcal{X}, \mathcal{Y}, \mathcal{Z}, T}
f^{\mathbb{Q}}(\mathcal{X}, \mathcal{Y}, \mathcal{Z}, T)
\]
is equivalent to the real optimization problem: $\min_{\mathcal{X}^R, \mathcal{Y}^R, \mathcal{Z}^R, T}
f^{\mathbb{R}}(\mathcal{X}^R, \mathcal{Y}^R, \mathcal{Z}^R, T),$
where $\mathcal{X}^R = (\mathcal{X})^R$, 
$\mathcal{Y}^R = (\mathcal{Y})^R$, 
$\mathcal{Z}^R = (\mathcal{Z})^R$, and
\begin{align*}
f^{\mathbb{R}}(\mathcal{X}^R, \mathcal{Y}^R, \mathcal{Z}^R, T)
&=
\frac{1}{4}\sum_{i=1}^{r}
\|(\mathcal{Y}^{R})_i\|_*^{\mathbb{R}}
+
\frac{\alpha}{8}
\bigl(\|\mathcal{X}^R - \mathcal{Z}^R \times_3 T^T\|_{F}^{\mathbb{R}}\bigr)^2
\\
&\quad
+
\frac{\beta}{8}
\bigl(\|\mathcal{Y}^R - \phi(\mathcal{Z}^R)\|_{F}^{\mathbb{R}}\bigr)^2
+
\Xi^{R}(\mathcal{X}^R)
+
\Xi(T),
\end{align*}
with $\mathcal{X}^R \in \mathbb{R}^{4n_1 \times 4n_2 \times n_3}$,
$\mathcal{Y}^R \in \mathbb{R}^{4n_1 \times 4n_2 \times r}$,
$\mathcal{Z}^R \in \mathbb{R}^{4n_1 \times 4n_2 \times r}$,
and $T \in \mathbb{R}^{r \times n_3}$.
\end{theorem}

\begin{proof}
Under the full real structure–preserving embedding, the Frobenius norm satisfies $\bigl(\|\mathcal{A}^R\|_{F}^{\mathbb{R}}\bigr)^2
=
4\bigl(\|\mathcal{A}\|_{F}^{\mathbb{Q}}\bigr)^2,$ and the quaternion nuclear norm satisfies $\|\mathcal{Y}_i\|_*^{\mathbb{Q}}
=
\frac14 \|(\mathcal{Y}^{R})_i\|_*^{\mathbb{R}}.$
Moreover, since $T$ is real-valued, the mode-3 product commutes with the real embedding:
\[
(\mathcal{X} \times_3 T)^R
=
\mathcal{X}^R \times_3 T.
\]
Therefore, each term in $f^{\mathbb{Q}}$ corresponds (up to positive constant scaling factors) to its real-domain counterpart in $f^{\mathbb{R}}$.
Since the real embedding is bijective, minimizing $f^{\mathbb{Q}}$ over quaternion variables is equivalent to minimizing $f^{\mathbb{R}}$ over their real embeddings.
\end{proof}

\begin{remark}
Throughout this section, we denote
$\Theta^k = (\mathcal{X}^k, \mathcal{Y}^k, \mathcal{Z}^k, T^k)$
for quaternion variables and
$\Theta^{R,k} = (\mathcal{X}^{R,k}, \mathcal{Y}^{R,k}, \mathcal{Z}^{R,k}, T^k)$
for their real structure–preserving representations,
with the understanding that $T^k$ is already real-valued.
\end{remark}


We now establish that each PAM iteration decreases the objective function while controlling the step size, which is the first key ingredient for convergence.

\begin{lemma}[Sufficient Decrease]
\label{lem:sufficient_decrease}
Let $\Theta^k = (\mathcal{X}^k, \mathcal{Y}^k, \mathcal{Z}^k, T^k)$ 
be the sequence generated by Algorithm~\ref{alg:qnttnn_real}. Then:
\[
f^{\mathbb{Q}}(\Theta^{k+1})
+
\rho \|\Theta^{k+1} - \Theta^k\|_{F}^{\mathbb{Q}\,2}
\leq
f^{\mathbb{Q}}(\Theta^k),
\text{~~where $\rho = \frac{1}{2}\min\{\rho_1, \rho_2, \rho_3, \rho_4\}$.}
\]
\end{lemma}

\begin{proof}
We analyze each update in the real structure–preserving embedding.

Denote $\Theta^{R,k} = (\mathcal{X}^{R,k}, \mathcal{Y}^{R,k}, \mathcal{Z}^{R,k}, T^k)$.

\textbf{Update $\mathcal{X}^{k+1}$:}
By optimality of the $\mathcal{X}$-subproblem, we have
\[
f^{\mathbb{R}}(\Theta^{R,k+1}|_{\mathcal{X}})
+
\frac{\rho_1}{8}
(\|\mathcal{X}^{R,k+1} - \mathcal{X}^{R,k}\|_{F}^{\mathbb{R}})^2
\leq
f^{\mathbb{R}}(\Theta^{R,k}).
\]

\textbf{Update $\mathcal{Y}^{k+1}$:}
By optimality of the real SVT step,
\[
f^{\mathbb{R}}(\Theta^{R,k+1}|_{\mathcal{Y}})
+
\frac{\rho_2}{8}
(\|\mathcal{Y}^{R,k+1} - \mathcal{Y}^{R,k}\|_{F}^{\mathbb{R}})^2
\leq
f^{\mathbb{R}}(\Theta^{R,k+1}|_{\mathcal{X}}).
\]

\textbf{Update $\mathcal{Z}^{k+1}$:}
Since the $\mathcal{Z}$-subproblem is minimized exactly,
\[
f^{\mathbb{R}}(\Theta^{R,k+1}|_{\mathcal{Z}})
+
\frac{\rho_3}{8}
(\|\mathcal{Z}^{R,k+1} - \mathcal{Z}^{R,k}\|_{F}^{\mathbb{R}})^2
\leq
f^{\mathbb{R}}(\Theta^{R,k+1}|_{\mathcal{Y}}).
\]

\textbf{Update $T^{k+1}$:}
The Procrustes step gives
\[
f^{\mathbb{R}}(\Theta^{R,k+1})
+
\frac{\rho_4}{2}
\|T^{k+1} - T^k\|_{F}^{2}
\leq
f^{\mathbb{R}}(\Theta^{R,k+1}|_{\mathcal{Z}}).
\]

Summing these four inequalities telescopically yields
\begin{align*}
f^{\mathbb{R}}(\Theta^{R,k+1})
&+
\frac{\rho_1}{8}
(\|\mathcal{X}^{R,k+1} - \mathcal{X}^{R,k}\|_{F}^{\mathbb{R}})^2
+
\frac{\rho_2}{8}
(\|\mathcal{Y}^{R,k+1} - \mathcal{Y}^{R,k}\|_{F}^{\mathbb{R}})^2
\\
&+
\frac{\rho_3}{8}
(\|\mathcal{Z}^{R,k+1} - \mathcal{Z}^{R,k}\|_{F}^{\mathbb{R}})^2
+
\frac{\rho_4}{2}
\|T^{k+1} - T^k\|_{F}^{2}
\leq
f^{\mathbb{R}}(\Theta^{R,k}).
\end{align*}

Using $(\|\mathcal{A}^R\|_{F}^{\mathbb{R}})^2
=
4(\|\mathcal{A}\|_{F}^{\mathbb{Q}})^2,$ we obtain the quaternion-domain inequality and the result follows.
\end{proof}


The second key ingredient relates the subdifferential size to the iterative progress, which is essential for applying the K-Ł inequality.

\begin{lemma}[Relative Error]
\label{lem:relative_error}
Assume the nonlinearity $\phi$ is Lipschitz continuous on bounded sets.
For any $d^{R,k+1} \in \partial f^{\mathbb{R}}(\Theta^{R,k+1})$, we have
\[
\|d^{R,k+1}\|_{F}^{\mathbb{R}}
\leq
\tau
\|\Theta^{R,k+1} - \Theta^{R,k}\|_{F}^{\mathbb{R}},
\]
where $\tau > 0$ is a constant depending on $\rho_i$ and the Lipschitz constants of the smooth parts.
\end{lemma}

\begin{proof}
By Lemma~\ref{lem:sufficient_decrease}, the sequence $\{\Theta^{R,k}\}$ satisfies sufficient decrease and is bounded. Hence the smooth part of $f^{\mathbb{R}}$ has Lipschitz continuous gradient on this bounded set.

For each block update, the first-order optimality condition gives
\begin{align*}
0 &\in \partial_{\mathcal{X}^R} f^{\mathbb{R}}(\Theta^{R,k+1})
+ \frac{\rho_1}{4}(\mathcal{X}^{R,k+1} - \mathcal{X}^{R,k}), \\
0 &\in \partial_{\mathcal{Y}^R} f^{\mathbb{R}}(\Theta^{R,k+1})
+ \frac{\rho_2}{4}(\mathcal{Y}^{R,k+1} - \mathcal{Y}^{R,k}), \\
0 &\in \partial_{\mathcal{Z}^R} f^{\mathbb{R}}(\Theta^{R,k+1})
+ \frac{\rho_3}{4}(\mathcal{Z}^{R,k+1} - \mathcal{Z}^{R,k}), \\
0 &\in \partial_T f^{\mathbb{R}}(\Theta^{R,k+1})
+ \rho_4(T^{k+1} - T^k).
\end{align*}

Using Lipschitz continuity of the gradient of the smooth part, we obtain
\[
\|d^{R,k+1}\|_{F}^{\mathbb{R}}
\leq
\tau
\|\Theta^{R,k+1} - \Theta^{R,k}\|_{F}^{\mathbb{R}}.
\]
\end{proof}

The third key ingredient is the K-Ł property, which we establish via the semi-algebraic structure of the objective function.

\begin{definition}[Kurdyka--Łojasiewicz (KŁ) Property]
Let $g : \mathbb{R}^m \to (-\infty,+\infty]$ be a proper lower
semi-continuous function.
We say that $g$ satisfies the Kurdyka--Łojasiewicz (KŁ) property
at $\bar{x} \in \mathrm{dom}(\partial g)$ if there exist
$\eta>0$, a neighborhood $\mathcal{U}$ of $\bar{x}$,
and a concave function
$\varphi : [0,\eta) \to \mathbb{R}_+$ such that:

\begin{itemize}
\item $\varphi(0)=0$,
\item $\varphi$ is $C^1$ on $(0,\eta)$ with $\varphi'>0$,
\item for all $x \in \mathcal{U}$ satisfying
\(
g(\bar{x}) < g(x) < g(\bar{x})+\eta,
\)
the inequality holds:
\[
\varphi'(g(x)-g(\bar{x}))
\,
\mathrm{dist}(0,\partial g(x))
\ge 1.
\]
\end{itemize}

If $g$ satisfies the KŁ property at every point of
$\mathrm{dom}(\partial g)$, then $g$ is called a KŁ function.
\end{definition}

\begin{theorem}[KŁ Property of $f^{\mathbb{R}}$]
\label{thm:KL_property}
The real objective function
$f^{\mathbb{R}}(\Theta^{R})$
satisfies the Kurdyka--Łojasiewicz property at every point in
$\mathrm{dom}(\partial f^{\mathbb{R}})$.
\end{theorem}

\begin{proof}
We verify that $f^{\mathbb{R}}$ is a definable function in an
o-minimal structure, which guarantees the KŁ property
\cite{attouch2013convergence}.
\begin{enumerate}
    \item 
\textbf{Nuclear norm term:}
For each real matrix $(\mathcal{Y}^{R})_i$,
the nuclear norm
\[
\|(\mathcal{Y}^{R})_i\|_*^{\mathbb{R}}
=
\sum_j \sigma_j((\mathcal{Y}^{R})_i)
\]
is a continuous semi-algebraic function.
Indeed, singular values are eigenvalues of
$((\mathcal{Y}^{R})_i)^T(\mathcal{Y}^{R})_i$,
whose characteristic polynomial is polynomial in the matrix entries.
Hence the nuclear norm is semi-algebraic.
\item \textbf{Frobenius norm terms:}
The quadratic terms
\[
(\|\mathcal{X}^{R}-\mathcal{Z}^{R}\times_3 T^T\|_F^{\mathbb{R}})^2
\quad \text{and} \quad
(\|\mathcal{Y}^{R}-\phi(\mathcal{Z}^{R})\|_F^{\mathbb{R}})^2
\]
are compositions of polynomial functions and smooth element-wise functions.
The mode-3 product is bilinear in $(\mathcal{Z}^{R},T)$,
hence polynomial.
Therefore the first quadratic term is polynomial.

\item \textbf{Nonlinearity:}
If $\phi$ is chosen as $\tanh$, sigmoid, or any real analytic function,
then $\phi$ is definable in the o-minimal structure
$\mathbb{R}_{\mathrm{an},\exp}$.
Hence the composition $\phi(\mathcal{Z}^{R})$ applied element-wise
is definable.

\item \textbf{Indicator functions:}
The constraint $\Xi^{R}(\mathcal{X}^{R})
=
\mathbf{1}_{\{\mathcal{P}_\Omega(\mathcal{X}^{R})
=
\mathcal{P}_\Omega(\mathcal{M}^{R})\}}$
is the indicator of an affine subspace,
defined by linear equations.
Similarly, $\Xi(T)
=
\mathbf{1}_{\{T:TT^T=I_r\}}$ is defined by polynomial equations.
Both sets are semi-algebraic.
\end{enumerate}
Since $f^{\mathbb{R}}$ is a finite sum and composition of
semi-algebraic and definable functions,
it is definable in an o-minimal structure.
By \cite{attouch2013convergence},
every proper lower semi-continuous definable function
satisfies the KŁ property.

\end{proof}


We now combine the three key ingredients to establish the main convergence result for our QNTTNN algorithm.

\begin{theorem}[Global Convergence of Algorithm~\ref{alg:qnttnn_real}]
\label{thm:global_convergence}
Assume $\phi$ is Lipschitz continuous on bounded sets. 
Then the sequence $\{\Theta^k\}$ generated by Algorithm~\ref{alg:qnttnn_real} satisfies:

\begin{enumerate}
    \item
     $\displaystyle{\sum_{k=0}^{\infty}
    \|\Theta^{k+1} - \Theta^k\|_{F}^{\mathbb{Q}}
    < +\infty.}$
    
    \item
    There exists $\Theta^*$ such that
    $\displaystyle{
    \lim_{k \to \infty} \Theta^k = \Theta^*,
    \qquad
    0 \in \partial f^{\mathbb{Q}}(\Theta^*)}$.
    
    \item
    The rate depends on the K-Ł exponent $\theta \in [0,1)$:
    \[
    \|\Theta^k - \Theta^*\|_{F}^{\mathbb{Q}}
    =
    \begin{cases}
    0, & \text{if } \theta = 0 \text{ (finite termination)}, \\
    O(\gamma^k), & \text{if } \theta \in (0, 1/2] \text{ (linear)}, \\
    O\!\left(k^{-\frac{1}{1-2\theta}}\right),
    & \text{if } \theta \in (1/2, 1) \text{ (sublinear)}.
    \end{cases}
    \]
\end{enumerate}
\end{theorem}

\begin{proof}
We apply the abstract proximal alternating minimization
convergence framework of \cite{attouch2013convergence}.

\noindent
\textbf{Step 1:} From Lemma~\ref{lem:sufficient_decrease}, the sequence satisfies
\[
f^{\mathbb{R}}(\Theta^{R,k+1})
+
\rho
\|\Theta^{R,k+1} - \Theta^{R,k}\|_{F}^{\mathbb{R}\,2}
\le
f^{\mathbb{R}}(\Theta^{R,k}),
\]
which gives sufficient decrease. From Lemma~\ref{lem:relative_error}, the relative error condition holds:
\[
\mathrm{dist}\!\left(
0, \partial f^{\mathbb{R}}(\Theta^{R,k})
\right)
\le
\tau
\|\Theta^{R,k} - \Theta^{R,k-1}\|_{F}^{\mathbb{R}}.
\]

From Theorem~\ref{thm:KL_property},
$f^{\mathbb{R}}$ satisfies the K-Ł property.

\noindent \textbf{Step 2: Finite length.} By sufficient decrease,
\[
\rho
\sum_{k=0}^{\infty}
\|\Theta^{R,k+1} - \Theta^{R,k}\|_{F}^{\mathbb{R}\,2}
\le
f^{\mathbb{R}}(\Theta^{R,0})
-
\inf f^{\mathbb{R}}
<
+\infty.
\]
Hence $\sum_{k=0}^{\infty}
\|\Theta^{R,k+1} - \Theta^{R,k}\|_{F}^{\mathbb{R}}
< +\infty$ by the K-Ł inequality (finite length property). Using $(\|\mathcal{A}^R\|_{F}^{\mathbb{R}})^2
=
4(\|\mathcal{A}\|_{F}^{\mathbb{Q}})^2,$ the result transfers to the quaternion domain.

\noindent \textbf{Step 3: Convergence to a critical point:} Since $\{\Theta^{R,k}\}$ has finite length, it is Cauchy and thus converges to some limit $\Theta^{R,*}$.
By the relative error condition,
\[
\mathrm{dist}\!\left(
0, \partial f^{\mathbb{R}}(\Theta^{R,k})
\right)
\to 0.
\]
Closedness of the limiting subdifferential yields $0 \in \partial f^{\mathbb{R}}(\Theta^{R,*}).$
By the bijectivity of the real embedding,
$0 \in \partial f^{\mathbb{Q}}(\Theta^{*}).$

\noindent\textbf{Step 4: Convergence rate:} Applying the standard K-Ł inequality
\[
\varphi'\!\left(
f^{\mathbb{R}}(\Theta^{R,k}) - f^{\mathbb{R}}(\Theta^{R,*})
\right)
\mathrm{dist}\!\left(
0, \partial f^{\mathbb{R}}(\Theta^{R,k})
\right)
\ge 1,
\]
and combining with sufficient decrease and the relative error bound,
the convergence rates follow from the general theory in
\cite{attouch2013convergence,bolte2014proximal}.
\end{proof}

\begin{remark}
If $\phi$ is real analytic (e.g., $\tanh$),
then $f^{\mathbb{R}}$ is definable in an o-minimal structure,
hence satisfies the K-Ł property with some exponent
$\theta \in [0,1)$.
The precise value of $\theta$ depends on the geometry of the problem
and cannot, in general, be determined explicitly.
\end{remark}

\section{Numerical Experiments}

In this section, we conduct numerical experiments on color video data to evaluate the performance of the proposed Quaternion Nonlinear Transform-based Tensor Nuclear Norm (QNTTNN) method for low-rank tensor completion (LRTC). Unlike conventional approaches that process RGB color channels independently, our quaternion-based framework represents each color pixel as a pure quaternion $q = 0 + R \mathbf{i} + G \mathbf{j} + B \mathbf{k}$, naturally encoding inter-channel correlations through quaternion algebra. This structure-preserving representation enables our method to leverage color relationships that are fundamentally lost in component-wise processing. All experimental video tensors are prescaled to $[0,1]$, and experiments are implemented in MATLAB R2023b running on Windows 11, using an Intel(R) Core(TM) $i9$-$12900$K ($3.2$~GHz) processor with 32~GB RAM.

We compare QNTTNN against three baseline methods: the t-SVD-based TNN \cite{zhang2017exact}, the DFT-based NTTNN~\cite{benzheng2022nonlineartrans} method applied channel-wise, and our proposed quaternion extension. Following established practices, we employ linear interpolation initialization for $\mc{X}^0$ and derive the transform initialization ${T}^0$ from the first $r$ left-singular vectors of the mode-3 unfolding of $\mc{X}^0$. The quaternion tensor is then initialized as $\mathcal{Z}^0 = \text{fold}_3({T}^0 \mc{X}^0_{(3)})$, with $\mathcal{Y}^0 = \phi(\mathcal{Z}^0)$ where $\phi(\cdot)$ denotes the hyperbolic tangent nonlinearity:
\begin{equation*}
\phi(x) = \frac{e^x - e^{-x}}{e^x + e^{-x}}.
\end{equation*}
Proximal parameters are set to $\rho_i = 0.001$ ($i = 1,2,3,4$), penalty parameters $\alpha$ and $\beta$ are selected from $\{1, 10, 100\}$, and the transform dimension $r$ is chosen from $\{3,4,5,6,7,8,9,10\}$ based on grid search optimization.
 We evaluate on \texttt{carphone\_qcif.y4m} ($144 \times 176 \times 100$) with random masks at sampling rates 5\%, 10\%, 20\%, 30\%.

Let $\mathcal{X}_0 \in \mathbb{Q}^{n_1 \times n_2 \times n_3}$ denote the original quaternion tensor with complete entries, and let $\mathcal{T} \in \mathbb{Q}^{n_1 \times n_2 \times n_3}$ be the observed tensor with missing entries. Denote by $\hat{\mathcal{X}} \in \mathbb{Q}^{n_1 \times n_2 \times n_3}$ the recovered tensor obtained by the proposed method. The reconstruction performance is evaluated using the relative squared error (RSE), defined as
\begin{equation*}
\mathrm{RSE} := \frac{\lVert \hat{\mathcal{X}} - \mathcal{X}_0 \rVert_F}{\lVert \mathcal{X}_0 \rVert_F},
\end{equation*}
and the peak signal-to-noise ratio (PSNR), defined as
\begin{equation*}
\mathrm{PSNR} := 10 \log_{10} \left( \frac{n_1 n_2 n_3 \, (m_1 - m_2)^2}{\lVert \hat{\mathcal{X}} - \mathcal{X}_0 \rVert_F^2} \right),
\end{equation*}
where $m_1$ and $m_2$ denote the maximum and minimum values of $\mathcal{X}_0$, respectively. Moreover, the average structural similarity index (SSIM) over all frames is adopted to assess the perceptual quality for color video reconstruction.

\subsection{Results}

\begin{table}[!ht]
\centering
\caption{Observed and Recovered PSNR/SSIM for Color Video under Different Sampling Rates}
\label{tab:video_results}
\begin{tabular}{|l|c|c|c|}
\hline
Method & Sampling = 5\% & Sampling = 10\% & Sampling = 30\% \\
\hline
Observed (Input) 
& 6.88 / 0.0146 
& 7.11 / 0.0236 
& 8.20 / 0.0571 \\

TNN~\cite{zhang2017exact}      
& 25.12 / 0.7190 
& 27.16 / 0.7871 
& 31.68 / 0.8985 \\
NTTNN~\cite{benzheng2022nonlineartrans}    
& 27.16 / 0.8078 
& 29.44 / 0.8707 
& 32.85 / 0.9340 \\
\textbf{QNTTNN} 
& \textbf{27.60 / 0.8228}
& \textbf{29.73 / 0.8723}
& \textbf{33.90 / 0.9372} \\
\hline
\end{tabular}
\end{table}

Table~\ref{tab:video_results} presents the quantitative comparison of QNTTNN against baseline methods across different sampling rates. The results demonstrate that QNTTNN consistently outperforms both TNN and channel-wise NTTNN across all sampling rates. At 5\% sampling, QNTTNN achieves 27.60 dB PSNR, representing a 0.44 dB improvement over NTTNN and 2.48 dB over TNN. The performance gap widens at higher sampling rates, with QNTTNN reaching 33.90 dB at 30\% sampling a 1.05 dB gain over NTTNN. The SSIM values further confirm the superior perceptual quality of quaternion-based recovery, maintaining values above 0.82 even at the challenging 5\% sampling rate. These improvements directly stem from the quaternion framework's ability to preserve inter-channel color correlations that are inherently discarded in component-wise processing approaches.

\begin{figure}[!h]
\centering
\begin{tabular}{@{}c@{\hspace{0.8mm}}c@{\hspace{0.8mm}}c@{\hspace{0.8mm}}c@{\hspace{0.8mm}}c@{}}

\includegraphics[width=0.18\textwidth]{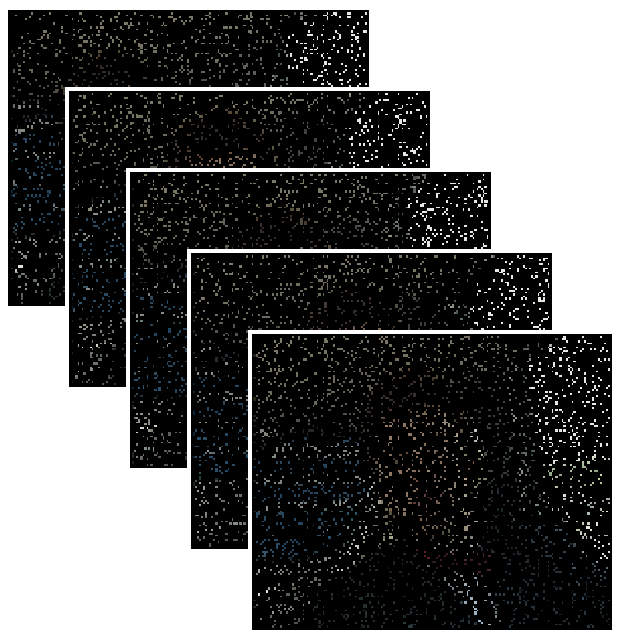} &
\includegraphics[width=0.18\textwidth]{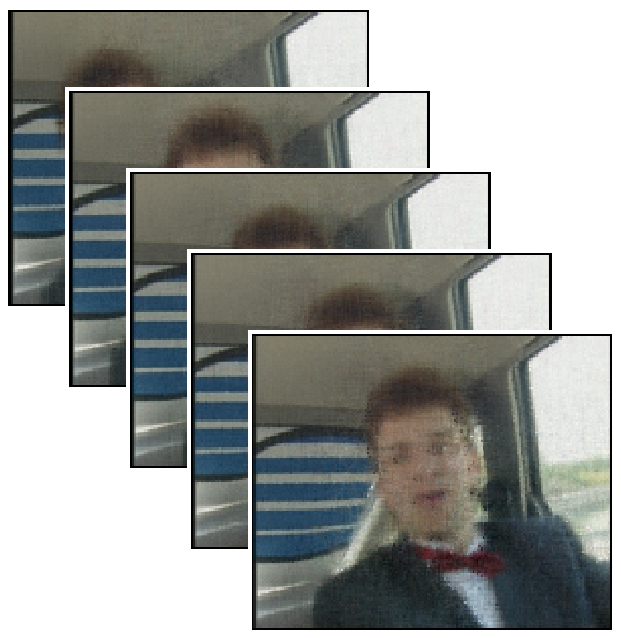} &
\includegraphics[width=0.18\textwidth]{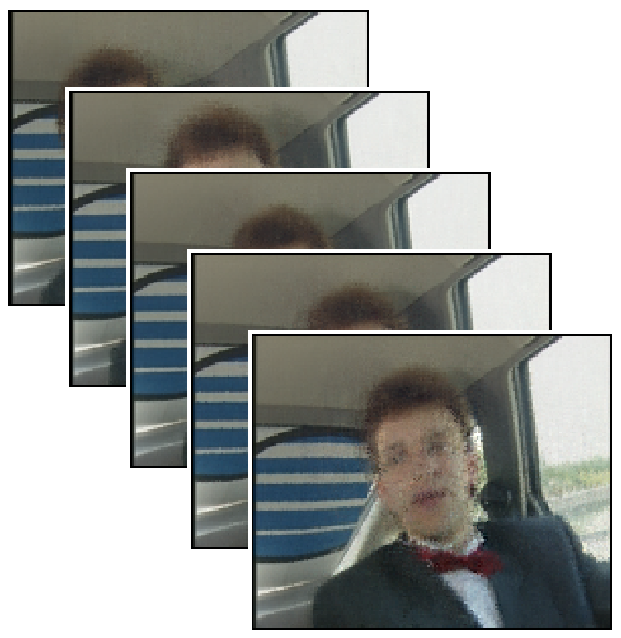} &
\includegraphics[width=0.18\textwidth]{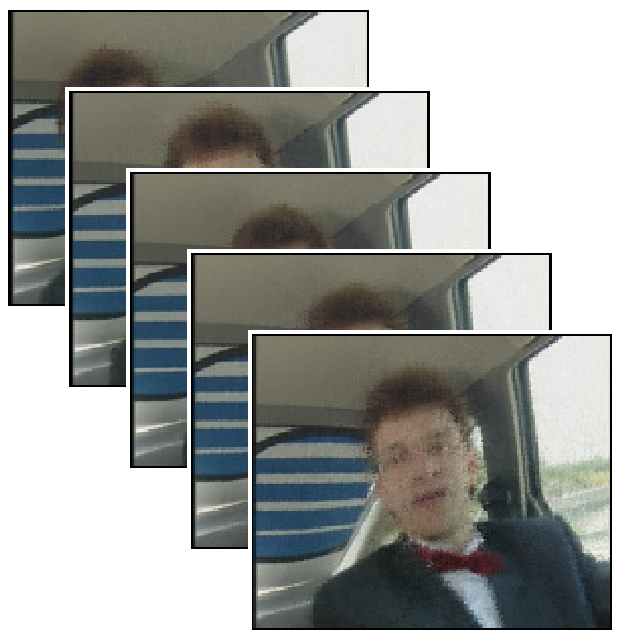} &
\includegraphics[width=0.18\textwidth]{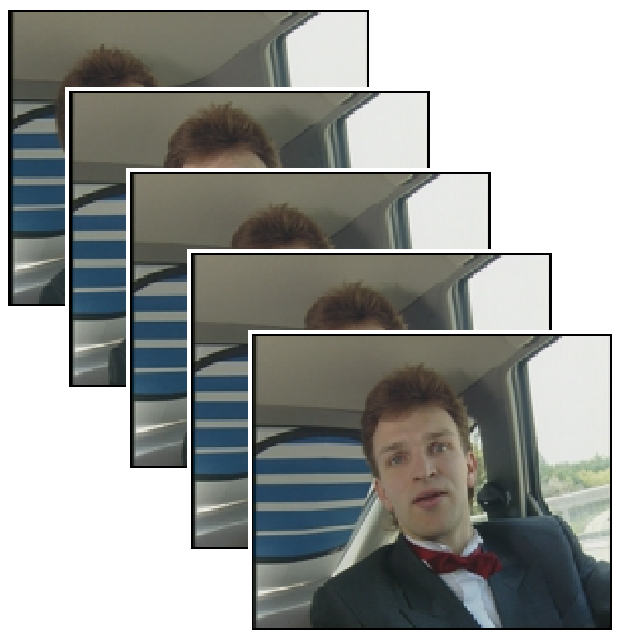} \\

\makebox[0.18\textwidth][c]{\scriptsize Observed} &
\makebox[0.18\textwidth][c]{\scriptsize TNN} &
\makebox[0.18\textwidth][c]{\scriptsize NTTNN} &
\makebox[0.18\textwidth][c]{\scriptsize \textbf{QNTTNN}} &
\makebox[0.18\textwidth][c]{\scriptsize Ground Truth}

\end{tabular}
\caption{Visual comparison of quaternion video recovery at Sampling Rate = 30\%.}
\end{figure}

\subsection{Performance versus Sampling Ratio}

Fig.~\ref{fig:sampling_ratio} illustrates the quantitative performance of different tensor completion methods with respect to the sampling ratio, evaluated in terms of PSNR, SSIM, and relative error. As the sampling ratio increases, all methods exhibit improved reconstruction quality, which is consistent with the availability of more observed entries. Notably, the proposed QNTTNN consistently achieves the highest PSNR and SSIM values and the lowest relative error across all sampling ratios, demonstrating its superior reconstruction capability. Compared with NTTNN and the linear TNN model, QNTTNN shows more pronounced performance gains in the low-sampling regime, indicating that the nonlinear transform in the quaternion domain is particularly effective in exploiting structural correlations under severely missing data. These results further validate the advantage of combining nonlinear transforms with quaternion-valued low-rank modeling for tensor completion.
\begin{figure}[t]
\centering
\includegraphics[width=0.32\linewidth]{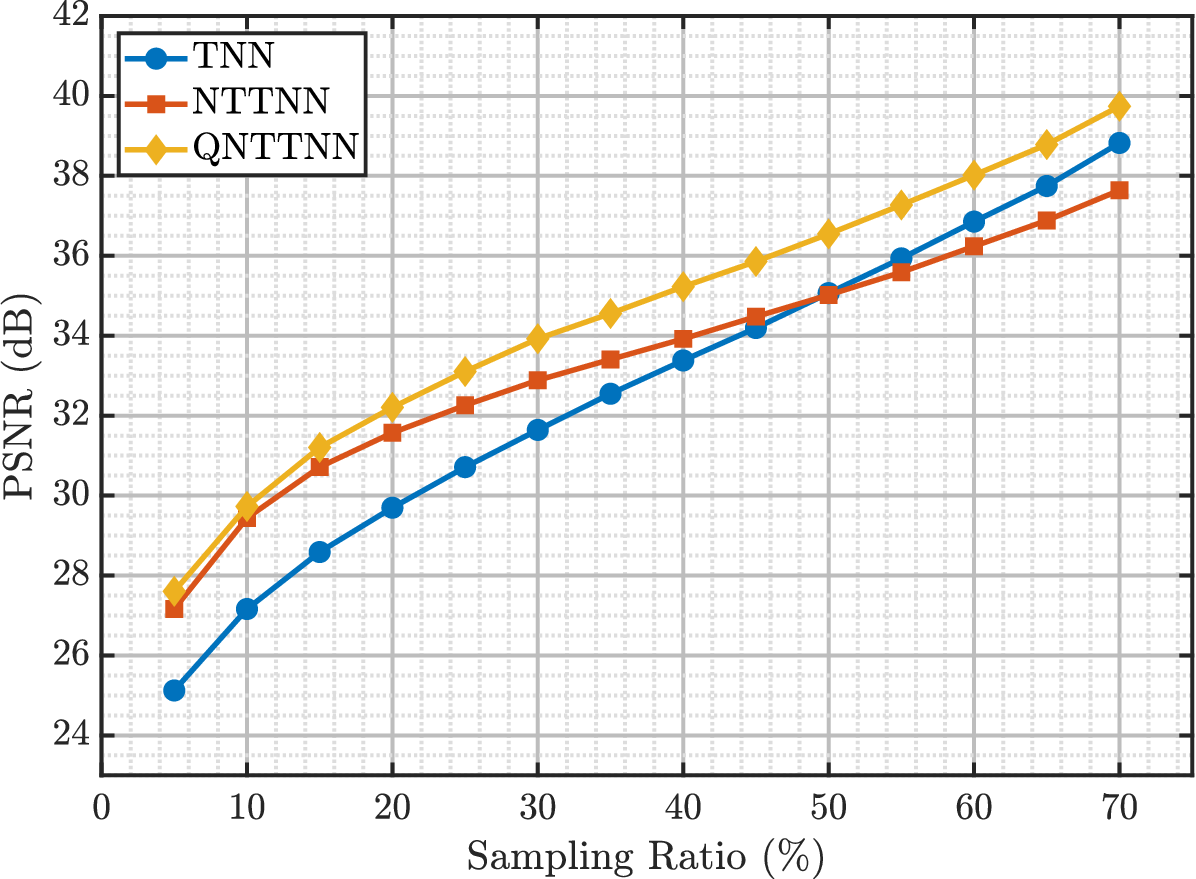}
\includegraphics[width=0.32\linewidth]{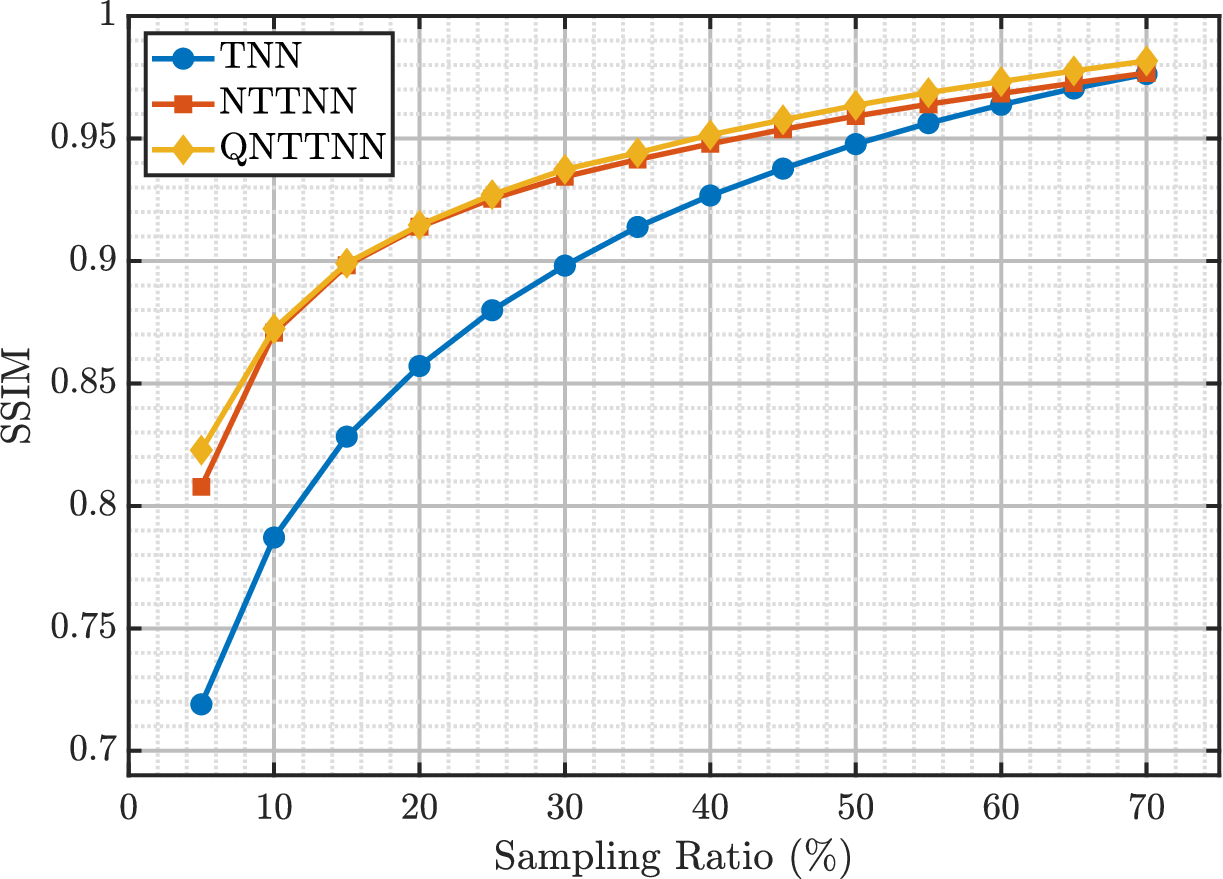}
\includegraphics[width=0.32\linewidth]{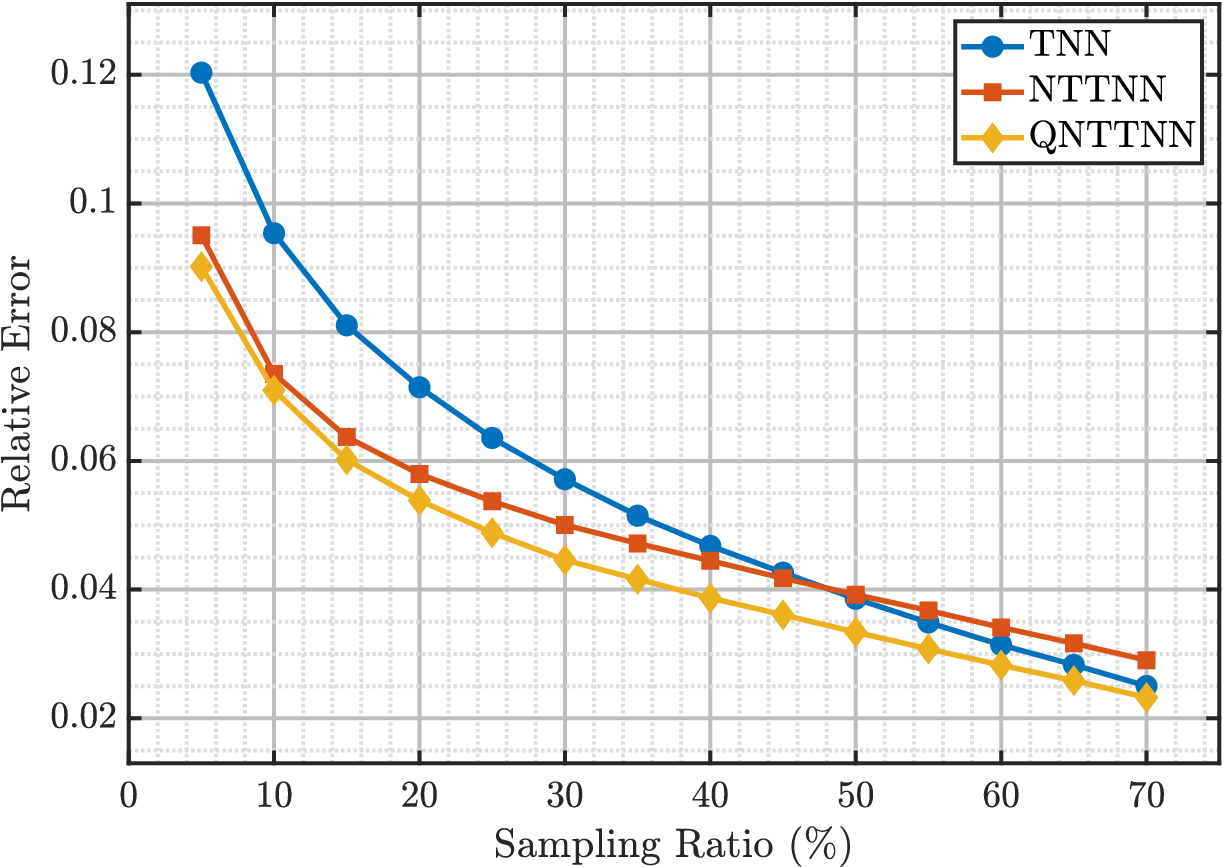}
\caption{Performance comparison with respect to sampling ratio: (left) PSNR, (middle) SSIM, and (right) relative error.}
\label{fig:sampling_ratio}
\end{figure}

\subsection{Nonlinear Transform Functions and Ablation Study}

In the proposed QNTTNN framework, an element-wise nonlinear transform is incorporated into the composite transform to enhance low-rankness in the transformed domain. In this work, we consider five commonly used nonlinear activation functions, whose mathematical definitions are summarized in Table~\ref{tab:nonlinear_defs}.

\begin{table}[t]
\centering
\caption{Definitions of nonlinear transform functions}
\label{tab:nonlinear_defs}
\setlength{\tabcolsep}{6pt}
\begin{tabular}{ll}
\hline
Transform & Definition \\
\hline
tanh & $\tanh(x) = \frac{e^x - e^{-x}}{e^x + e^{-x}}$ \\
sigmoid & $\sigma(x) = \frac{1}{1 + e^{-x}}$ \\
ReLU & $\mathrm{ReLU}(x) = \max(0, x)$ \\
ELU & $\mathrm{ELU}(x) =
\begin{cases}
x, & x>0, \\
\alpha(e^x - 1), & x \le 0,
\end{cases}$ \\
Swish & $\mathrm{Swish}(x) = x \cdot \sigma(x)$ \\

\hline
\end{tabular}
\end{table}

To investigate the impact of different nonlinear transforms in the proposed QNTTNN framework, we conduct an ablation study by replacing the element-wise nonlinear function in the composite transform with each of the functions listed in Table~\ref{tab:nonlinear_defs}. All experiments are performed under identical experimental settings, with a fixed sampling ratio of $10\%$ and the same algorithmic parameters, to ensure a fair comparison.

\begin{table}[t]
\centering
\caption{Ablation study on nonlinear transform functions in QNTTNN}
\label{tab:nonlinear_ablation}
\setlength{\tabcolsep}{6pt}
\begin{tabular}{lccc}
\hline
Transform & PSNR (dB) & SSIM & Time (s) \\
\hline
tanh    & \textbf{29.73} & \textbf{0.8723} & 114.4 \\
sigmoid & 29.22 & 0.8690 & 116.2 \\
ReLU    & 25.06 & 0.7217 & 126.3 \\
ELU     & 29.16 & 0.8644 & 151.2 \\
Swish   & 27.79 & 0.8298 & 116.1 \\

\hline
\end{tabular}
\end{table}

Table~\ref{tab:nonlinear_ablation} reports the quantitative results in terms of PSNR, SSIM, and computational time. Among the evaluated nonlinear functions, \textit{tanh} achieves the best overall performance, yielding the highest PSNR (29.73~dB) and SSIM (0.8720), which demonstrates its strong capability in enhancing low-rankness under the nonlinear transform. \textit{Sigmoid} and \textit{ELU} also exhibit competitive performance, with slightly lower reconstruction accuracy but comparable stability. In contrast, \textit{ReLU} shows significantly inferior performance, which can be attributed to its non-smooth and one-sided nature that weakens low-rank promotion in the transformed domain. \textit{Swish} provides moderate improvements but does not outperform bounded nonlinearities such as \textit{tanh} and \textit{sigmoid}.

Overall, these results indicate that smooth and bounded nonlinear transforms are more effective in the proposed QNTTNN framework, as they better preserve structural correlations and promote low-rankness in quaternion frontal slices.

\section{Conclusion}
In this paper, we proposed a quaternion nonlinear transform and developed the corresponding Quaternion Nonlinear Transform–Induced Tensor Nuclear Norm (QNTTNN) for quaternion-valued tensor completion. The proposed framework employs a composite transform consisting of a linear semi-orthogonal transform along the third mode and an element-wise nonlinear transform applied to quaternion frontal slices via real embedding, enabling effective exploitation of inter-channel correlations. Based on the proposed quaternion low-rank metric, we formulated a quaternion tensor completion model and designed an efficient proximal alternating minimization algorithm, whose convergence to critical points is theoretically guaranteed by Kurdyka–Łojasiewicz analysis in the real embedded domain. Extensive experiments on benchmark color video inpainting datasets demonstrate that Q-NTTNN consistently outperforms channel-wise NTTNN .

As future work, we plan to extend the proposed QNTTNN framework to learning-driven quaternion models, such as integrating adaptive nonlinear transforms and quaternion neural networks, and to explore its applications in large-scale video restoration and multimodal quaternion-valued data analysis.




\section*{Conflict of Interest}
The authors declare that they have no competing interests.





\section*{Funding}

Ratikanta Behera is supported by the Anusandhan National Research Foundation (ANRF), Government of India, under Grant No. EEQ/2022/001065.

\end{document}